\documentclass[11pt,twoside]{amsart}
\usepackage{amsfonts,amsmath,amstext,amsbsy,amsthm,amscd,amsmath,amssymb,dsfont}
\usepackage[utf8]{inputenc}
\usepackage{hyperref}
\usepackage{comment}
\usepackage[T1]{fontenc}
\usepackage{color}

\newtheorem{Lemme}{Lemma}[section]
\newtheorem{Prop}{Proposition}[section]  

\newtheorem{Rmq}{Remark}[section]
\newtheorem{Thm}{Theorem}[section]

\theoremstyle{remark}

\usepackage{indentfirst}
\newcommand{\be}{\begin{equation}}
\newcommand{\ee}{\end{equation}}
\newcommand{\ba}{\begin{array}}
\newcommand{\ea}{\end{array}}
\newcommand{\bea}{\begin{eqnarray}}
\newcommand{\eea}{\end{eqnarray}}
\newcommand{\bee}{\begin{eqnarray*}}
\newcommand{\eee}{\end{eqnarray*}}
\newcommand{\rob}{\textcolor{blue}}

\newcommand{\R} {\mathbb{R}}    
    
\newcommand{\Z} {\mathbb{Z}}

\newcommand{\cF} {\mathcal{F}}     
     
\newcommand{\cI} {\mathcal{I}}

\newcommand{\cR} {\mathcal{R}}     
\newcommand{\cS} {\mathcal{S}}

\newcommand{\La} {\Lambda}

\def \with {\quad\!\hbox{with}\!\quad}
\def \andf {\quad\!\hbox{and}\!\quad}

\def\dZ_1{\delta\!Z_1}

\def\d{\partial}

\def\wh{\widehat}
\def\wt{\widetilde}
\def\pa{\partial}
\def\ddj{\dot\Delta_j}
\def\ddk{\dot\Delta_k}
\def\ddq{\dot\Delta_q}

\usepackage[foot]{amsaddr}

\title[]{Relaxation approximation and asymptotic stability \\of stratified solutions to the IPM equation}

\date{}

\usepackage[notcite,notref]{showkeys}


\author[R. Bianchini]{Roberta Bianchini}
\address[R. Bianchini]{Consiglio Nazionale Delle Ricerche, 00185, Rome, Italy}
\author[T. Crin-Barat]{Timothée Crin-Barat}
\address[T. Crin-Barat]{Chair of Computational Mathematics, Fundación Deusto,	Avenida de las Universidades, 24, 48007 Bilbao, Basque Country, Spain.}
\author[M. Paicu]{Marius Paicu}
\address[M. Paicu]{Institut de Mathématiques de Bordeaux, Université de Bordeaux, 33405 Talence Cedex, France}
\begin{document}
\maketitle

\begin{abstract}
 We prove the nonlinear asymptotic stability of stably stratified solutions to the Incompressible Porous Media equation (IPM) for 
initial perturbations in $\dot H^{1-\tau}(\R^2) \cap \dot H^s(\R^2)$ with $s > 3$ and for any $0 < \tau <1$. Such result improves the existing literature, where the asymptotic stability is proved for initial perturbations belonging at least to $H^{20}(\R^2)$. \\
 More precisely, the aim of the article is threefold.
 First, we provide a simplified and improved proof of global-in-time well-posedness of the Boussinesq equations with strongly damped vorticity in $H^{1-\tau}(\R^2) \cap \dot H^s(\R^2)$ with $s > 3$ and $0 < \tau <1$. Next, we prove the strong convergence of the Boussinesq system with damped vorticity towards (IPM) under a suitable scaling.  Lastly, the asymptotic stability of stratified solutions to (IPM) follows as a byproduct.\\
%The aim of the paper is threefold. First, we provide a new  and improved proof of existence of unique global smooth solutions in $\R^2$ to the 2d inviscid Boussinesq equations with a damping term in the velocity equation. The next result is a rigorous justification of the relaxation limit of the 2d inviscid Boussinesq system with damped velocity, in the regime of infinite damping, towards the Incompressible Porous Media (IPM) equation. 
A symmetrization of the approximating system and a careful study of the anisotropic properties of the equations via anisotropic Littlewood-Paley decomposition play key roles to obtain uniform energy estimates. Finally, one of the main new and crucial points is the integrable time decay of the vertical velocity $\|u_2(t)\|_{L^\infty (\R^2)}$ for initial data only in $\dot H^{1-\tau}(\R^2) \cap \dot H^s(\R^2)$ with $s >3$.

%Finally, as a byproduct, we provide new results on the global smooth solutions to the (IPM) equations linearized around stratified steady states in $\mathbb{R}^2$.
\end{abstract}

% \tableofcontents

\section{Introduction}\label{sec:intro}
The Incompressible Porous Media (IPM) system in two space dimensions is an active scalar equation 
\begin{equation}\label{eq:IPMsystem}
\begin{cases}
        \d_t \eta + \textbf{u} \cdot \nabla \eta=0,\\
        \textbf{u}=- \kappa \nabla P + \textbf{g}\eta, \qquad \textbf{g}=(0,-g)^T, \qquad \text{(Darcy law)} \\
        \nabla \cdot \textbf{u}=0,
\end{cases}
\end{equation}
modeling the dynamics of a fluid of density $\eta=\eta(t, x, y):\R^+\times \R^2 \rightarrow \R$ through a porous medium according to the Darcy law, where $\kappa >0 $ and $g>0$ are the permeability coefficient and the gravity acceleration respectively, which hereafter are assumed to be $\kappa=g=1$.
We refer to \cite{castro2} and references therein for further explanations on the physics and the applications of the model. The active scalar velocity $\textbf{u}=(u_1, u_2)$ of system \eqref{eq:IPMsystem} can be reformulated in terms of 
a singular integral operator of degree 0 as follows
\begin{equation}
    \textbf{u}=- (\cR_2, - \cR_1) \mathcal{R}_1 \rho, 
\end{equation}
where $(\cR_1, \cR_2)$ is the two-dimensional Riesz transform, i.e.
\begin{align}\label{eq:riesz}
    \cR_1=(-\Delta)^{-1/2} \partial_x, \qquad  \cR_2=(-\Delta)^{-1/2} \partial_y.
\end{align}
We are interested in the stability properties of the stratified steady state $\overline \rho_{\text{eq}}(y)=\rho_0-y$, where $\rho_0>0$ is a constant averaged density. 
Introducing the perturbed unknown $\rho=\rho(t,x,y)$ such that $\eta (t,x,y)=\overline{\rho}_{\rm{eq}}(y)+\rho(t,x,y)$, the perturbation $\rho$ satisfies
\begin{equation}\label{eq:IPM}
\begin{aligned}
    \pa_t \rho - \cR_1^2 \rho  = (\cR_2 \cR_1\rho , -\cR_1^2 \rho) \cdot \nabla \rho.
\end{aligned}
 \tag{IPM-diss}
\end{equation}
The nonlinear asymptotic stability of the stratified steady state $\overline{\rho}_{\text{eq}}(y)=\rho_0-y$ in the whole space $\R^2$ has been first established by Elgindi in \cite{ElgindiIPM}, for initial data at least in $H^{20}(\R^2).$
The analogous result but in the periodic finite channel $\mathbb{T}\times [-\pi,\pi]$ is due to Castro, C\'ordoba and Lear \cite{castro2} under slightly less restrictive regularity assumptions. 
We remark that the linear operator 
$$\d_t \rho -\cR_1^2 \rho=0$$
 in frequency space $\xi=(\xi_1, \xi_2)\in \R^2 $ is dissipative everywhere but the hyperplane $\xi_1=0$ thanks to the monotonicity $\bar \rho_{\text{eq}}'(y)<0$ (stable stratification, see \cite{gallay, lannes}). 
Therefore, it appears natural to approximate \eqref{eq:IPM} by a system with a (partially) dissipative linear part. In this regard, it turns out that the two-dimensional Boussinesq equations in vorticity form with strongly damped vorticity is a good approximation to \eqref{eq:IPM} under a suitable scaling of time and unknowns. 
To the best of our knowledge, the approximation of \eqref{eq:IPM} via the Boussinesq system with damped vorticity is new and establishing its rigorous validity in the sense of strong convergence of solutions is in the scope of the present work.\\
The two-dimensional Boussinesq equations with damped velocity with gravity acceleration $g=1$ and background density profile $\overline \rho_{\text{eq}}(y)=\rho_0-y$  read as follows
\begin{equation}\label{eq:B-vel}
\left\{
\begin{aligned}
\d_t b- u_2 &=-(u_1\d_x b + u_2\d_y b), \\
\d_t u_1 + \d_x P &=-\frac{u_1}{\varepsilon}-(u_1\d_x u_1 + u_2\d_y u_1), \\
\d_t u_2 + \d_y P &=-b-\frac{u_2}{\varepsilon}-(u_1\d_x u_2 + u_2\d_y u_2), \\
\d_x u_1 + \d_y u_2 &=0,
\end{aligned}
\right.
\end{equation}
where $\textbf{u}=(u_1,u_2) \in \R^2$ is the velocity field, $b$ is the buoyancy term and $P$ is the incompressible pressure.% that we additionally (but reasonably, see \cite{castro2019} and references therein) assume to be constant. 
 \:
This system is obtained by a linearization of the density-dependent incompressible Euler equations with damped velocity around the \emph{hydrostatic stratified steady state}
\begin{align}\label{eq:rest}
(\overline \rho_{\rm{eq}} (y), (0,0), \overline P_{\rm{eq}}(y)) \quad \text{with} \quad \overline P_{\rm{eq}}'(y)=- \overline \rho_{\rm{eq}} (y).
\end{align}

 Several mathematical works in the existing literature have been devoted to this system. Besides explaining the applications of \eqref{eq:B-vel} to electrocapillarity, Castro, C\'ordoba and Lear in \cite{castro2019} (see also references therein) establish the global-in-time well-posedness of system \eqref{eq:B-vel} in the bounded domain $\mathbb{T} \times [0,1]$.
In the whole space, global smooth solutions have been first provided by Wan in \cite{WanBoussinesq} by means of Green function analysis and energy methods.

We refer to \cite{BianchiniNatalini2DB, castro2019} for a (formal) derivation of \eqref{eq:B-vel} by a linearization of the density-dependent incompressible Euler equations around the hydrostatic stratified steady state under the strong Boussinesq approximation $\rho \sim \rho_0$ \cite{lannes}.
\\
For our scopes, it is convenient to rewrite the system in vorticity $\omega$ and stream function $\phi$ formulation where $\textbf{u}=\nabla^\perp \phi$ with the sign convention (as in \cite{ElgindiIPM})  $\nabla^\perp=(\d_y, -\d_x)$, which gives 
\begin{equation}
\left\{
\begin{aligned}
\label{eq:system-oldvariable}
    &\d_t b - (-\Delta)^{-1} \d_x \omega  = - \textbf{u} \cdot \nabla b, \\
    &\d_t \omega +  \frac{\omega}{\varepsilon} - \d_x b  = - \textbf{u} \cdot \nabla \omega,\\
    &\Delta \phi=\omega.
\end{aligned}
\right.\tag{2D-Bouss}
\end{equation}
Our first goal is to establish a systematic and improved proof of the global-in-time existence of smooth solutions for small data to \eqref{eq:system-oldvariable} (in Theorem \ref{Thm:Exist}),
 exploiting the anisotropic nature of the system by means of anisotropic Littlewood-Paley decomposition of the Fourier space $(\xi_1, \xi_2) \in \R^2$ as introduced in \cite{CheminZhangAniso, PaicuAniso, klaus}. Since the \emph{linear} and \emph{undamped} approximation of the system supports the propagation of anisotropic waves of dispersion relation ${\pm \xi_1}/{|\xi|}= \mp i\widehat{\cR_1}$ (see \cite{bianchini2021, lannes}), it is not surprising that the horizontal anisotropic decomposition of the phase space plays a key role in our refined analysis.
In our case, the linear part of system \eqref{eq:system-oldvariable} is dissipative provided that $\xi_1 \neq 0$, therefore it is natural to build an energy functional with an anisotropic Fourier multiplier as a weight (multiplier method).
This idea allows to prove that $\|u_2(t)\|_{L^\infty(\R^2)}$ is integrable in time without assuming any additional $L^1(\R^2)$ integrability of the initial data. More precisely, the control of $\|u_2\|_{L^1_T L^\infty (\R^2)}$ without $L^1$-integrability and high regularity assumptions is new and this is a striking point to provide a substantially improved global-in-time well-posedness of \eqref{eq:system-oldvariable} for small data only in $\dot H^{1-\tau}(\R^2) \cap \dot H^s(\R^2) $ with $0 < \tau <1$ and $s \ge 3+\tau$.
 
For convenience of the reader, we point out that our small-data global-in-time existence in Theorem \ref{Thm:Exist} could be reformulated as a result of \emph{nonlinear asymptotic stability} of the hydrostatic steady state \eqref{eq:rest}  with $\overline{\rho}_{\rm{eq}}(y)=\rho_0-y$, under the evolution of the Boussinesq system below 
\begin{equation}\label{eq:euler}
    \begin{cases}
          \partial_t \eta + \textbf{u}\cdot \nabla \eta=0, \\
          \partial_t  \textbf{u} +  \textbf{u} \cdot \nabla  \textbf{u} + \nabla P =\eta \textbf{g}, \qquad \textbf{g}=(0,-g),\\
          \nabla \cdot \textbf{u}=0.\qquad\qquad \qquad\qquad \qquad \qquad  
    \end{cases}\tag{E}
\end{equation}
This reformulation holds for system \eqref{eq:euler} with initial density $\eta_{\rm{in}}(x,y)$ such that
$\|\eta_{\rm{in}}-\overline{\rho}_{\rm{eq}}\|_X \ll 1$ for a suitable functional space $X$, in the same spirit of \cite{castro2019}.
 
Our second and main goal is to rigorously justify the relaxation limit of the two-dimensional Boussinesq equations with damped vorticity \eqref{eq:system-oldvariable} towards \eqref{eq:IPM} under a suitable scaling that we introduce later on. To the best of our knowledge, this relaxation approximation is new. 

Finally, as a byproduct of the relaxation limit and the global well-posedness of \eqref{eq:system-oldvariable}, we provide a new proof of existence of global smooth solutions in $\mathbb{R}^2$ to the equation \eqref{eq:IPM} for the perturbation $\rho$, with small initial data $\rho (0,x,y)=\rho_{\rm{in}}(x,y) \in \dot H^{1-\tau}(\R^2) \cap \dot H^s(\R^2)$ with $0< \tau <1$ and $s \ge 3+\tau$.
Agaib, this result can be reformulated in terms of the solution $\eta$ to \eqref{eq:IPMsystem}, for an initial datum $\eta_{\rm{in}}(x,y)=\overline{\rho}_{\rm{eq}}(y)+\rho_{\rm{in}}(x,y)$ such that $\|\eta_{\rm{in}}(x,y)-\overline{\rho}_{\rm{eq}}(y)\|_X =\|\rho_{\rm{in}}(x,y)\|_{\dot H^{1-\tau} \cap \dot H^s} \ll 1$. Such reformulation yields the nonlinear asymptotic stability of equation \eqref{eq:IPMsystem} around the stratified steady state $\overline{\rho}_{\rm{eq}}(y)$, namely the setting of Elgindi \cite{ElgindiIPM} that proves the result in $H^{20}(\R^2)$, while in our Theorem \ref{Thm:ExistIPM} we only need that the initial perturbation $\rho_\text{in}(x,y) \in \dot H^{1-\tau}(\R^2) \cap \dot H^s$ with $s>3$.
%%%%%

\subsection{A new formulation}
A first key element of our approach is the use of the symmetrized variables introduced in 
\cite{BianchiniNatalini2DB}. With the notation 
\begin{align*}
    \Lambda=(-\Delta)^{1/2}=\sqrt{\xi_1^2+\xi_2^2},
\end{align*}
where $\xi=(\xi_1, \xi_2) \in \mathbb{R}^2$ denotes the frequency coordinate in Fourier space, 
we introduce the new unknown
\begin{equation}
\begin{aligned}\label{def:Omega}
    \Omega:= \Lambda^{-1} \omega,
\end{aligned}
\end{equation}
so that system \eqref{eq:system-oldvariable} rewrites as
\begin{equation}
\left\{
\begin{aligned} \label{eq:system-newvariable}
    &\d_t b -  \cR_1 \Omega = (\cR_2 \Omega, - \cR_1 \Omega) \cdot (\nabla b), \\
    &\d_t \Omega + \frac{\Omega}{\varepsilon} - \cR_1 b = \Lambda^{-1} [ (\cR_2 \Omega, - \cR_1 \Omega) \cdot (\nabla \Lambda\Omega)],
\end{aligned}\tag{2D-B}
\right.
\end{equation}
where $\cR_j$, $j\in\{1,2\}$ are the components of the Riesz transform as in \eqref{eq:riesz}.
%  associated with the 0-th pseudodifferential symbol $i\Lambda^{-1}$
Now, we introduce the auxiliary variable
\begin{align}\label{def:z}
z:=  \Omega - \varepsilon \cR_1 b. 
\end{align}
The system in $(b, z)$ then reads as follows
\begin{equation}
\left\{
\begin{aligned}
    &\d_t b - \varepsilon\cR_1^2 b  = \cR_1 z + (\cR_2 \Omega, - \cR_1 \Omega) \cdot \nabla b, \\
    &\d_t z + \frac{z}{\varepsilon}={-}\varepsilon \cR_1^2 \Omega {-} \varepsilon \cR_1 [(\cR_2 \Omega, -\cR_1\Omega) \cdot \nabla b] + \Lambda^{-1} [(\cR_2 \Omega, -\cR_{\rob{1}}\Omega) \cdot (\nabla \Lambda \Omega)]. \label{eq:system-z2}
\end{aligned}
\right.
\end{equation}

Such formulation partially diagonalizes the linear part of the system, except two linear terms in the right-hand side of \eqref{eq:system-z2} that in the energy estimates are absorbed by the left-hand side. This allows us to avoid hypocoercivity techniques \cite{BZ,CBD2}, obtaining a priori estimates simply based on the energy method.
We point out that these a priori estimates are \emph{uniform} in the vanishing parameter $\varepsilon$, which is a key point to justify the relaxation towards \eqref{eq:IPM}.

 The use of the good unknown $z=\Omega - \varepsilon\cR_1 b$ is inspired both by the work of Hoff and Haspot in \cite{Hoff,Boris}, where the authors introduce the effective velocity for the compressible Navier-Stokes equations and the results of Crin-Barat and Danchin in \cite{CBD1,CBD3,ThesisCB} on partially dissipative hyperbolic systems.
In the aforementioned works, this reformulation is only possible in some specific frequency regime (high frequencies for Navier-Stokes and low frequencies for partially dissipative systems) where the eigenvalues of the linear operator are real. In the present paper, the particular form of the Riesz transform (an operator of degree 0) allows us to use such diagonalization in all frequencies as the eigenvalues $\lambda_\pm$ of the linear part of system \eqref{eq:system-newvariable}
\begin{align*}
    \lambda_{\pm}=\frac{1}{2 \varepsilon} \pm \frac{1}{2 \varepsilon } \sqrt{1-\frac {4 \varepsilon^2 \xi_1^2}{|\xi|^2}}
\end{align*}
are real in the whole frequency space for any $\varepsilon \le 1/2$, so that the variable $z$ in \eqref{def:z} is a good unknown in all frequency regimes.

%We also mention the work of Crin-Barat, He and Shou \cite{HPC-CBHS} concerning the Hyperbolic-Parabolic chemotaxis system where two damped modes have to be used to diagonalize properly the system under consideration. 

\subsection{Formal justification of the relaxation limit as $\varepsilon\to0$ and main results.}
Taking inspiration from the theory of partially dissipative systems (see for instance \cite{BHN, mar0,CoulombelGoudon,mar1,hsiao1,Junca,XuWang,CoulombelLin} and references therein), one can expect to rigorously justify the relaxation limit from \eqref{eq:system-newvariable} towards \eqref{eq:IPM} as $\varepsilon \to 0$, by applying the following scaling:
\begin{equation} \label{DiffuRescale}
(\widetilde{b}^\varepsilon,\widetilde{\Omega}^\varepsilon)(\tau,x)
\triangleq(b,\frac{\Omega}{\varepsilon})(t,x)
\with \tau=\varepsilon t.
\end{equation}
The system \eqref{eq:system-newvariable} in the scaled unknowns  $(\widetilde{b}^\varepsilon,\widetilde{\Omega}^\varepsilon)$ reads as follows:
\begin{equation}\label{eq:bouss}
\left\{
\begin{aligned}
    \pa_t \widetilde{b}^\varepsilon - \cR_1 \widetilde{\Omega}^\varepsilon & = (\cR_2 \widetilde{\Omega}^\varepsilon) \pa_x \widetilde{b}^\varepsilon - (\cR_1 \widetilde{\Omega}^\varepsilon) \pa_y \widetilde{b}^\varepsilon, \\
    \varepsilon^2\pa_t \widetilde{\Omega}^\varepsilon - \cR_1 \widetilde{b}^\varepsilon + \widetilde{\Omega}^\varepsilon&=\varepsilon^2\Lambda^{-1} [(\cR_2 \wt \Omega^\varepsilon, - \cR_1 \wt\Omega^\varepsilon) \cdot (\nabla \Lambda\wt\Omega^\varepsilon)].
\end{aligned}
\right.
\end{equation}

$\bullet$ Our first result is a systematic and improved proof of the global-in-time well-posedness, with uniform estimates in the relaxation parameter $\varepsilon \leq \frac{1}{2}$, of the above system with initial data $(b_{\rm{in}}, \Omega_{\rm{in}}) \in \dot H^{1-\tau} \cap \dot H^s$ for any $0 < \tau < 1$ and $s \ge 3+\tau$. The approach is based on the use of the anisotropic Littlewood-Paley decomposition that allows to capture the (anisotropic) nature of the equation in a nearly optimal way. 
For references on the use of anisotropic Besov spaces in the analysis of incompressible flows we refer to \cite{klaus} and we mention \cite{elgindi2} for a study of the effect of anisotropy (the Riesz transform) in low regularity.
\vspace{3mm}

Sending $\varepsilon\to0$ in system \eqref{eq:bouss}, one formally obtains that $\widetilde \Omega^\varepsilon \to \Omega$ and $\widetilde b^\varepsilon \to \rho$, where $\rho$ satisfies the Incompressible Porous Media equation \eqref{eq:IPM} and the Darcy law
\begin{align}\label{eq:rho}
    \Omega+\cR_1 \rho=0.
\end{align}

We improperly call \emph{diffusive scaling} the change of coordinates \eqref{DiffuRescale}. Of course it is not the usual diffusive scaling $(t/\varepsilon^2, x/\varepsilon, y/\varepsilon)$ under which the heat equation is invariant, but inserting \eqref{eq:rho} into the linear part of the equation for $\widetilde b^\varepsilon$ in \eqref{eq:bouss} yields 
\begin{align*}
    \partial_t \rho - \cR_1^2 \rho =0,
\end{align*}
where $\cR_1^2=(-\Delta)^{-1} \partial_{xx}$ is a partially \emph{dissipative operator}.

$\bullet$ Our next result is a mathematical proof of the relaxation limit of $\eqref{eq:bouss}$ towards \eqref{eq:IPM} as $\varepsilon \to 0$.
%using the uniform-in-$\varepsilon$ estimates of the solution to $\eqref{eq:bouss}$ (provided by the proof of Theorem \ref{Thm:Exist}). 
To the best of our knowledge, the relaxation approximation of \eqref{eq:IPM} provided by  \eqref{eq:system-newvariable} is new as well as its rigorous justification.

\vspace{3mm}
On the limit system \eqref{eq:IPM}, the existence of global-in-time smooth solutions to \eqref{eq:IPM} with small data (or, equivalently, the asymptotic stability of \eqref{eq:IPMsystem} around the stratified steady state) has been first proved by Elgindi in \cite{ElgindiIPM}, both in the full space $\R^2$ and in $\mathbb{T}^2$. In particular,
Elgindi shows that the stratified steady state $\overline{\rho}_{\rm{eq}}(y)=\rho_0-y$ is asymptotically stable in ${ H}^s(\R^2)$, for $s\ge 20$, with the additional integrability assumption that initial perturbations $\rho_{\text{in}}(x,y)$ belong to the $L^1$-based Sobolev space $W^{4,1}$ (this is also a key hypothesis to use dispersion effects in \cite{klaus1}).
Although being groundbreaking, the result in \cite{ElgindiIPM} requires very high regularity of solutions. 

$\bullet$ As a byproduct of the relaxation limit, in this article we provide a new proof of existence of global-in-time smooth solutions to \eqref{eq:IPM}, only assuming that the initial datum  $\rho_{\text{in}}(x,y) \in \dot H^{1-\tau} \cap \dot H^s$ for any $0 < \tau < 1$ and $s \ge \tau+3$.

\medbreak\noindent{\bf Acknowledgments.}  TCB is partially  supported by the European Research Council (ERC) under the European Union's Horizon 2020 research and innovation programme (grant agreement NO: 694126-DyCon). RB is partially supported by the GNAMPA group of INdAM and the PRIN project 2020 \emph{Nonlinear evolution PDEs, fluid dynamics and transport equations: theoretical foundations and applications}.

RB thanks \'Angel Castro for several useful discussions on the (IPM) equation and Klaus Widmayer for pointing out the use of anisotropic Besov spaces in fluid-dynamics problems.

\section{Main results}
Recalling the notation $\La=(-\Delta)^\frac 12$, we introduce the homogeneous Sobolev space for $s \in \R$
\begin{align}
    \|f\|_{\dot H^s}= \|\La^s f\|_{L^2(\mathbb{R}^2)}=\left(\int_{\mathbb{R}^2} |\xi|^{2s} |\widehat{f}(\xi)|^2 \, d\xi\right)^\frac{1}{2}. 
\end{align}
We also use the notation, for any $s, s' \in \R$,
\begin{align}
    \|f\|_{\dot H^s\cap \dot H^{s'}} = \|f\|_{\dot H^s}+\|f\|_{\dot H^{s'}}.
\end{align}

Our main results hold in Sobolev spaces, however we will rely on the properties of \emph{anisotropic} Besov spaces to obtain some estimates that play a crucial role in our proof. Such anisotropic spaces allow to perfectly capture the anisotropic nature of the 2D Boussinesq system. Similar approaches have been applied to the MHD system by Lin and Zhang in \cite{LinZhang3DMHD,LinZhang2DMHD} and to the incompressible Navier-Stokes equations by Chemin and Zhang \cite{CheminZhangAniso} and Paicu \cite{PaicuAniso}. More recently, an approach based on anisotropic Besov spaces has been developed for 3D rotating incompressible fluids by Guo, Pausader and Widmayer \cite{klaus}.

 Our first result concerns the uniform global well-posedness of system \eqref{eq:system-oldvariable}. For any $r>0$, we define the following functional
 \begin{align}\label{def:Mmu}
  \mathcal{M}_r (T)&:=\|(b,\Omega,z)\|_{L^\infty_T({\dot H}^r)}+\sqrt\varepsilon \|\cR_1b\|_{L^2_T({\dot H}^r)}+ \frac{1}{\sqrt{\varepsilon}}\|\Omega\|_{L^2_T({\dot H}^r)}+\frac{1}{\sqrt\varepsilon} \|z\|_{L^2_T({\dot H}^r)}.
 \end{align}

%%%%%
 We obtain the following.
 \begin{Thm}[Global existence for \eqref{eq:system-newvariable}]\label{Thm:Exist}
For any $0<\varepsilon \le 1/2$ and any $0< \tau < 1$, let $s\ge 3+\tau$. For any couple of initial data $(b_{\rm{in}},\Omega_{\rm{in}})\in {\dot H}^{1-\tau}(\R^2) \cap  {\dot H}^s(\R^2)$,
there exists a constant value $0<\delta_0 \ll 1$ such that, under the following assumption
 \begin{align}\label{eq:smallness-thm}
     \mathcal{M}(0)=\|(b_{\rm{in}},\Omega_{\rm{in}})\|_{\dot H^{1-\tau}\cap\dot{H}^s}\leq \delta_0,
 \end{align}
there exists a unique global-in-time smooth solution $(b,\Omega)$ to system \eqref{eq:system-newvariable} satisfying the following inequality for all times $t>0$
 \begin{align}\label{def:X}
   X(t):=\mathcal{M}(t)+\|\nabla u_2\|_{L^1_T(L^\infty)}+\|\La u_2\|_{L^1_T(L^\infty)} \leq \mathcal{M}(0),
\end{align}
where 
\begin{align}
    u_2=(-\Delta)^{-1/2} \partial_x \Omega=\cR_1 \Omega,
\end{align}
and 
\begin{align}\label{def:M}
  \mathcal{M}(t):=\mathcal{M}_{1-\tau}(t)+\mathcal{M}_{s}(t).  
\end{align}
\end{Thm}

\begin{Rmq}[On the expression of $\mathcal{M}(t)$]
The functional $\mathcal{M}(t)$ in \eqref{def:M} is the sum of two terms, i.e. $\mathcal{M}_{1-\tau}(t)$ and $\mathcal{M}_{s}(t)$, defined in \eqref{def:Mmu}, which are both crucial for the embedding into anisotropic Besov spaces ${\dot H}^s \cap {\dot H}^{1-\tau}\subset B^{s_1, s_2}$ in Lemma \ref{BesovSobolev}. This is a key point: in fact, although the core of our analysis will be developed in anisotropic Besov spaces $B^{s_1, s_2}$ (introduced in Section \ref{sec:mainr}), however the final result is stated in Sobolev regularity precisely thanks to the embedding ${\dot H}^s \cap {\dot H}^{1-\tau} \subset B^{s_1, s_2}$.
\end{Rmq}

\begin{Rmq}[Our setting and comparison with the result of Wan \cite{WanBoussinesq}]
The global well-posedness of the 2D Boussinesq system with damping \eqref{eq:B-vel} in $\R^2$ was first established by Wan in \cite{WanBoussinesq}. Besides providing a more systematic and shorter proof (of Theorem \ref{Thm:Exist}) that exploits the anisotropic nature of the system, the present work improves several points.
\begin{itemize}
    \item The regularity assumptions are lowered: \cite{WanBoussinesq} requires $b \in \dot H^{-1} \cap \dot H^{s_0}, \, \omega \in \dot H^{-2} \cap \dot H^{s_0-1}$ with $s_0 \ge 6$, while here we only need $b \in \dot H^{1-\tau} \cap \dot H^{s}, \, \omega \in \dot H^{-\tau} \cap \dot H^{s-1} $ for any $0 <\tau < 1$ and $s \ge 3+\tau$.
    \item While \cite{WanBoussinesq} only bounds $\|u_2\|_{L^\frac{4}{3}_T  L^\infty}$ relying on spectral analysis and sophisticated estimates, here we provide a control of $\|u_2\|_{L^1_T L^\infty}$ thanks to the anisotropic Littlewood-Paley approach; notice that the $L^1_T$ control of $\|u_2\|_{L^\infty}$ with $u_2=\cR_1 \Omega$ is natural as $\|\cR_1 \Omega\|_{L^\infty}$ is expected to decay  at integrable rate like $\nabla u_2$ (see Proposition 3.4 in \cite{ElgindiIPM}, \cite{BianchiniNatalini2DB}, and Remark 4.2 in \cite{WanBoussinesq}).
    \item In stark contrast with \cite{WanBoussinesq} where $b \in \dot H^{-1}$ (see Remark 4.1 in \cite{WanBoussinesq}), we stress that here we do not need any unnatural assumption on $b$, which simply belongs to $\dot H^{1-\tau} \cap \dot H^s, \, s \ge 3+\tau, \; 0 < \tau < 1$.
    \item Finally, it is interesting to point out that in our framework $b \in \dot H^{1-\tau} \cap \dot H^{s}$ (homogeneous spaces with positive indexes) is not required to be square integrable, which is natural as both the Boussinesq equations \eqref{eq:system-oldvariable} and the incompressible porous media equation \eqref{eq:IPM} are invariant by the transformation $b\rightarrow b+C$, for any constant $C \in \R$.
\end{itemize}

% For completeness, we mention the work by Castro, C\'ordoba and Lear \cite{castro2019} that provides the well-posedness of the Boussinesq system with damping in the bounded domain $\mathbb{T} \times [0,1]$ (with boundary conditions). 
\end{Rmq}

\begin{Rmq}[On the $\dot H^{1-\tau}$ estimate]
As just remarked before, we only take $(b_\text{in}, \Omega_\text{in}) \in \dot H^{1-\tau} \cap \dot H^s$ without involving any negative Sobolev space (in \cite{WanBoussinesq} $b_\text{in} \in \dot H^{-1} \cap \dot H^{s_0}$ and $\omega_{\text{in}} \in \dot H^{-2} \cap H^{s_0-1}$ with $s_0 \ge 6$).  %Notice that the same holds for $\Omega \in \dot H^{1-\tau} \cap \dot H^s$. On the other hand, since $\Omega= (-\Delta)^{-1/2} \omega$ (see \eqref{def:Omega}), our assumption in terms of $\omega$ reads $\omega \in \dot H^{-\tau} \cap \dot H^{s-1}, s >3$. Again, this is a net improvement with respect to the setting of \cite{WanBoussinesq}, where $\omega_{\text{in}} \in \dot H^{-2} \cap H^{s_0-1}$ with $s_0 \ge 6$. 
However, we point out that there is no need of $\omega \in \dot H^{-\tau}$ (i.e. $\Omega \in \dot H^{1-\tau}$) in the first part, namely the proof of Proposition \ref{Prop:apriori-Sobolev}. The assumption  $\Omega_{\rm{in}} \in \dot H^{1-\tau}$ is only used in the anisotropic Besov part (more precisely in the control of the $Y(t)$ functional in \eqref{YFunc}) to bound $\|\Omega\|_{L^\infty_T(B^{\frac 12, \frac 12})}$ by means of the embedding (Lemma \ref{BesovSobolev}) ${\dot H}^s \cap {\dot H}^{1-\tau} \subset B^{\frac12,\frac12}$, and therefore to control $\mathcal{M}_{1-\tau}$.
%\textcolor{red}{We believe that this control (and then the assumption $(b_{\rm{in}}, \Omega_{\rm{in}}) \in \dot H^{1-\tau}$) could be removed using commutator estimates (instead of simple product rules in anisotropic Besov space) in the proof of Proposition \ref{Propu2}. However, this would cost a huge amount of technicality that we decided to avoid for readability of the manuscript. }
\end{Rmq}
%\textcolor{blue}{T. Not sure if we should keep this sentence. It is a bit vague, essentially the commutator are not available for non-summability reasons, which is why we do not use them, not really technicality}

The result below concerns the justification of the relaxation limit.

\begin{Thm}[Relaxation limit]\label{Thm:Relax}
 Let the hypotheses of Theorem \ref{Thm:Exist} be fulfilled and let $(\wt b^\varepsilon,\wt \Omega^\varepsilon)$ be the unique solution, scaled with \eqref{DiffuRescale}, associated to the initial data $(b_{\rm{in}},\Omega_{\text{in}})$ as in Theorem \ref{Thm:Exist}.

Then, for any $0<s'<s$ and $0<\tau<\tau'<1$, one has the limit as $\varepsilon\to0$, 
 $$\wt b^\varepsilon\to\rho \text{ strongly in }C([0,T],\dot{H}^{1-\tau'}_{loc}\cap\dot{H}^{s-s'}_{loc}),$$
 where $\rho$ is the unique solution of \eqref{eq:IPM} associated to the initial data $b_\text{in}$.
 Moreover, it holds
 $$\|\rho(\cdot,t)\|_{\dot H^{1-\tau}\cap\dot H^s}\leq C\|(b_{\rm{in}},\Omega_{\rm{in}})\|_{\dot H^{1-\tau}\cap\dot H^s},$$
 where the constant $C$ is independent of $\varepsilon$.\\
 Finally, we recover the Darcy law in the following sense:
 $$\|\wt\Omega^\varepsilon-\cR_1\wt b^\varepsilon\|_{L^1_T(B^{\frac32,\frac12}\cap B^{\frac12,\frac12})}\leq \varepsilon \mathcal{M}(0).$$
\end{Thm}
\begin{Rmq}[On the Darcy law]
Note that applying the operator $\nabla^\perp \cdot$ to the velocity $\textbf{u}$ in \eqref{eq:IPMsystem} with $\kappa=g=1$ (and replacing the notation $\eta$ by $\rho$) yields \begin{align*}
    \omega=\nabla^\perp \cdot \textbf{u}=\d_x \rho,
\end{align*}
which in terms of the variables $(\Omega, \rho)$ reads exactly $\Omega=\cR_1 \rho$.
\end{Rmq}

Our analysis also provides a new proof of existence of global-in-time smooth solutions to the incompressible porous media equation \eqref{eq:IPM} for small data.
 
 \begin{Thm}[Existence for \eqref{eq:IPM}]\label{Thm:ExistIPM}
For any $0< \tau < 1$, let $s\ge 3+\tau$. For any initial datum $\rho_{\rm{in}}\in {\dot H}^{1-\tau}(\R^2) \cap  {\dot H}^s(\R^2)$,
there exists a constant value $0<\delta_0 \ll 1$ such that, under the assumption
 \begin{align*}
     \|\rho_{\rm{in}}\|_{\dot H^{1-\tau}\cap\dot{H}^s}\leq \delta_0,
 \end{align*}
there exists a unique global-in-time smooth solution $\rho$ to system \eqref{eq:IPM} satisfying the following inequality for all times $t>0$
 \begin{align*}
   X(t):=\|\rho\|_{L^\infty_T(\dot H^{1-\tau}\cap\dot{H}^s)}+\|\cR_1\rho\|_{L^2_T(\dot H^{1-\tau}\cap\dot{H}^s)}+\|(\nabla \cR_1^2 \rho,\La \cR_1^2 \rho)\|_{L^1_T(L^\infty)}\lesssim \|\rho_{\rm{in}}\|_{\dot H^{1-\tau}\cap\dot{H}^s}.
\end{align*}
\end{Thm}

\begin{Rmq}[Comparison with the result of Elgindi \cite{ElgindiIPM}]
The global well-posedness of \eqref{eq:IPM} for small data, namely the asymptotic stability of \eqref{eq:IPMsystem} around the stratified steady state $\overline{\rho}_{\rm{eq}}(y)=\rho_0-y$,
 was first proved by Elgindi in \cite{ElgindiIPM} taking $b \in W^{4,1} \cap H^{s_0}, \, s_0\ge 20$.
 
Our Theorem \ref{Thm:ExistIPM} provides a new proof of Elgindi's result, which allows to substantially lower 
the regularity assumption of \cite{ElgindiIPM},
taking only $b \in \dot H^{1-\tau} \cap \dot H^s$ with $0 < \tau < 1$ and $s \ge 3+\tau$.
Once again, we take advantage of the anisotropic Littlewood-Paley decomposition and anisotropic Besov spaces that capture the time-integrability of the solution without relying on Green function estimates of the linearized problem. 
We also mention the asymptotic stability result by Castro-C\'ordoba-Lear \cite{castro2} of \eqref{eq:IPMsystem} around the stratified steady state $\overline{\rho}_{\text{eq}}=\rho_0-y$ in the domain $\mathbb{T} \times [-\pi, \pi].$
\end{Rmq}

 \begin{Rmq}[Instability results from Kiselev-Yao]
A consequence of the recent result \cite[Theorem 1.5]{KiselevYao} by Kiselev and Yao is that there exists an initial perturbation $\rho_\text{in}(x,y)$ satisfying
$\|\rho_\text{in}\|_{H^{2-\gamma}(\mathbb{T}\times [-\pi, \pi])}\ll 1$ for any $\gamma>0$, such that
the solution $\rho(t, x, y)$ to \eqref{eq:IPM} (provided it remains smooth for all times) displays the following time growth
 $$\lim \sup\limits_{t\to \infty} t^{-\frac s2 }\|\rho(t)\|_{\dot H^{s+1}(\mathbb{T}\times [-\pi, \pi])}=0,$$ for all $s>0$.
 In \cite[Remark 1.6]{KiselevYao}, the authors ask whether $\rho_{\text{in}}(x,y)$ can be made small in higher Sobolev spaces, while $\rho (t, x, y)$ still displays time growth. 
 Even though in the present work we study the case of the full space $\R^2$ rather than the bounded periodic channel, we underline that our Theorem \ref{Thm:ExistIPM} states that all perturbations $\rho_{\text{in}}$ that are small in $H^{s}(\R^2)$ with $s>3$ generate solutions $\rho (t, x, y)$ that remain small for all times. Thus, if a blow-up in finite time or a time-growth happens for solutions in the whole space $\R^2$, the initial perturbation must have less regularity than $H^s(\R^2)$, $s>3$.
\end{Rmq}
 
 \section{Anisotropic Besov spaces}
  
  \subsubsection{Anisotropic Littlewood-Paley decomposition}\label{sec:mainr}
 
 We introduce the following anisotropic Littlewood-Paley decompositions: for $j,q,k\in\Z$, we denote
\begin{itemize}
\item $\ddj$ the blocks associated to the Littlewood-Paley decomposition in $|\xi|$;
\item $\ddq^h$ the blocks associated to the Littlewood-Paley decomposition in the direction $\xi_1$,
\item $\ddk^v$ the blocks associated to the Littlewood-Paley decomposition in the direction $\xi_2$,

\end{itemize}
such that, denoting by $\mathcal{F}$ the Fourier transform,
$$\ddj u=\cF^{-1}(\varphi(2^{-j}|\xi|)\widehat{u}) \:\:\:\ddq^h u=\cF^{-1}(\varphi(2^{-q}\xi_1)\widehat{u}) \andf \ddk^v u=\cF^{-1}(\varphi(2^{-k}\xi_2)\widehat{u})  ,$$
where $\varphi(\xi) = \phi (\xi/2)-\phi(\xi)$ and $\phi \in C_c^\infty$ is such that $\phi =1 $ for $|\xi| \le 1/2$ and $\phi(\xi) = 0 $ for $|\xi| \ge 1$.
We define the following \emph{homogeneous} and \emph{anisotropic} Besov semi-norms:
\begin{align*}
\|f\|_{\dot B^{s}_{p,r}}&\triangleq \bigl\| 2^{js}\|\ddj f\|_{L^p(\R^d)}\bigr\|_{\ell^r(j\in\Z)},\\
\|f\|_{\dot B^{s_1,s_2}_{p,r}}&\triangleq \bigl\| 2^{js_1}2^{qs_2}\|\ddj\ddq^h f\|_{L^p(\R^d)}\bigr\|_{\ell^r(j\in\Z,k\in\Z)}.
\end{align*}

Hereafter we will omit the \emph{dot} (standing for \emph{homogeneous} spaces) and the second and third Besov indexes will be dropped as well for lightening the notation
$$\|f\|_{B^{s_1,s_2}}\triangleq\|f\|_{\dot B^{s_1,s_2}_{2,1}}.$$

%%%%%
\subsection{Technical results in anisotropic Besov spaces}\label{sec:para}
We now state an anisotropic version of Bernstein's lemma, the proof of which can be found in \cite{LinZhang3DMHD,PaicuAniso}.

\begin{Lemme}[Bernstein-type inequalities] \label{AnisoBernstein}
For $x=(x_1,x_2)\in\R^2$, let $B_1$ be a ball of $\R_{x_1}$, $B_2$ be a ball of $\R_{x_2}$, $C_1$ be an annulus of $\R_{x_1}$ and $C_2$ an annulus of $\R_{x_2}$. Let $1\leq p_1 \leq p_2 \leq \infty$ and $1\leq q_1 \leq q_2 \leq \infty$. Then, we have
\begin{itemize}
    \item If the support of $\wh a$ is included in $2^q B_1$, then
    \begin{align*}
        \|\d_{x_1}^s a\|_{L^{p_2}_{x_1}(L^{q_1}_{x_2})} \lesssim 2^{q(|s|+(\frac{1}{p_1}-\frac{1}{p_2}))}\| a\|_{L^{p_1}_{x_1}(L^{q_1}_{x_2})}.
    \end{align*}
    
   \item If the support of $\wh a$ is included in $2^k B_2$, then
    \begin{align*}
        \|\d_{x_2}^s a\|_{L^{p_1}_{x_1}(L^{q_2}_{x_2})} \lesssim 2^{k(|s|+2(\frac{1}{q_1}-\frac{1}{q_2}))}\| a\|_{L^{p_1}_{x_1}(L^{q_1}_{x_2})}.
    \end{align*}  
    
       \item If the support of $\wh a$ is included in $2^q C_1$, then
    \begin{align*}
        \| a\|_{L^{p_1}_{x_1}(L^{q_1}_{x_2})} \lesssim 2^{-q |s|}\|\d_{x_1}^s a\|_{L^{p_1}_{x_1}(L^{q_1}_{x_2})}.
    \end{align*} 
    
     \item If the support of $\wh a$ is included in $2^k C_2$, then
    \begin{align*}
        \| a\|_{L^{p_1}_{x_1}(L^{q_1}_{x_2})} \lesssim 2^{-k |s|}\|\d_{x_2}^s a\|_{L^{p_1}_{x_1}(L^{q_1}_{x_2})}.
    \end{align*} 
    
\end{itemize}

\end{Lemme}

 Embeddings of Sobolev spaces into anisotropic Besov spaces are provided by the result below.
\begin{Lemme}[Embedding in Sobolev space, \cite{LinZhang2DMHD}, Lemma 3.2]
\label{BesovSobolev}
Let $s_1,s_2,\tau_1,\tau_2\in\R$ such that $\tau_1<s_1+s_2<\tau_2$ and $s_2>0$. If $a\in\dot{H}^{\tau_1}(\R^2)\cap \dot{H}^{\tau_2}(\R^2)$ and $a\in B^{s_1,s_2}$, then
\begin{align*}
    \|a\|_{B^{s_1,s_2}}\lesssim     \|a\|_{B^{s_1+s_2}}\lesssim \|a\|_{\dot H^{\tau_1}}+\|a\|_{\dot H^{\tau_2}}.
\end{align*}
\end{Lemme}

We will rely on the embeddings $B^{\frac{3}{2},\frac{1}{2}} \hookrightarrow\text{Lip},$ $B^{\frac{1}{2},\frac{3}{2}}\hookrightarrow \text{Lip}(\cR_1 \cdot)$ and  $B^{-\frac{1}{2},\frac{5}{2}} \hookrightarrow \text{Lip}(\cR_1^2 \cdot)$, where for $n=1,2$, the notation $\text{Lip}(\cR_1^n\cdot)$ denotes the space of functions whose Riesz transform of order $n$ is Lipschitz, cf. the left-hand sides of \eqref{EmbeddingLipR1} for the associated norms.
 \begin{Lemme}[Embeddings in $\text{Lip}$]\label{AnisoEmbed}
Let $a\in B^{\frac{3}{2},\frac{1}{2}}\cap B^{\frac 12, \frac 12} \cap B^{\frac 12, \frac 32}\cap B^{-\frac 12, \frac 52}$.
% Let $a\in B^{\frac{3}{2},\frac{1}{2}}$, $\cR_1a \in B^{\frac 12, \frac 12} \cap B^{\frac 12, \frac 32}$ and  $\cR_1^2a\in B^{-\frac 12, \frac 52}$.
The following inequalities hold:
\begin{align*}
&\|a\|_{L^\infty}\lesssim  \|a\|_{B^{\frac 12,\frac12}}, \quad  \|\nabla  a \|_{L^\infty} \lesssim  \|a\|_{B^{\frac 32,\frac12}} \andf \|\La a\|_{L^\infty}\lesssim\|a\|_{B^{\frac 32, \frac 12}}.
\end{align*}
When the Riesz operator in the direction $x$ is involved, one has
\begin{align}\label{EmbeddingLipR1}
\|\nabla\cR_1 a \|_{L^\infty}\lesssim \|a\|_{B^{\frac 12,\frac32}},\quad \|\La \cR_1 a\|_{L^\infty} \lesssim \|a\|_{B^{\frac 12, \frac 32}}
 \andf \|\nabla\cR_1^2 a \|_{L^\infty}\lesssim \|a\|_{B^{-\frac 12,\frac52}}.
\end{align}
\end{Lemme}
\begin{proof}
Using the anisotropic Bernstein Lemma \ref{AnisoBernstein}, one obtains
\begin{align*}
     \|\nabla a \|_{L^\infty}\lesssim  \sum_{j,q\in\Z^2}2^{j}\|\ddj\ddq^h a \|_{L^\infty}
     &\lesssim  \sum_{j,q\in\Z^2,j\geq k}2^{j}\|\ddj\ddq^h\ddk^v a \|_{L^\infty}
     \\&\lesssim \sum_{j,q,k\in\Z^3,j\geq k}2^{j}2^{\frac{q}{2}}2^{\frac{k}{2}}\|\ddj\ddq^h\ddk^v a \|_{L^2}
     \\&\lesssim \sum_{j,q\in\Z^2}2^{\frac{3j}{2}}2^{\frac{q}{2}}\|\ddj\ddq^h a \|_{L^2}
     \\&\lesssim \|a\|_{B^{\frac 32,\frac12}_{2,1}}.
\end{align*}
When replacing the operator $\nabla$ by $\Lambda$, the proof follows exactly the same lines.
The estimates involving the Riesz operator $\cR_1$ can be obtained in a similar manner noticing that for $s\in\{1,2\}$ 
\begin{align*}
     \|\nabla\cR_1^s a \|_{L^\infty}\lesssim  \sum_{j,q\in\Z^2}2^{j}2^{-s j}2^{s q}\|\ddj\ddq^h a \|_{L^\infty}.
\end{align*}
\end{proof}

 \section{Proof of Theorem \ref{Thm:Exist}} \label{Sec:Exist}
 The proof of Theorem \ref{Thm:Exist} is divided in two main steps. We first provide the estimates in Sobolev spaces, namely we control the functional $\mathcal{M}(t)$ in \eqref{def:M}. 
The outcome of those estimates in Proposition \ref{Prop:apriori-Sobolev} involves the quantities $\|\nabla u_2\|_{L^1_T(L^\infty)}, \, \|\La u_2\|_{L^1_T(L^\infty)}$, whose control is not provided by the energy functional $\mathcal{M}(t)$.
 Thus, in a second step in Section \ref{sec:Linftycontrol}, we rely on \emph{anisotropic} Besov spaces to estimate the aforementioned quantities and to conclude the proof.
 %%%%%
 \subsection{I. Control of $\mathcal{M}(t)$}
 The estimates in homogeneous Sobolev spaces are provided by the result below.
\begin{Prop}\label{Prop:apriori-Sobolev}
 Let $(b,\Omega)$ be a smooth solution to system \eqref{eq:system-newvariable}. Then the following holds
 \begin{align*}
\mathcal{M}(T) \lesssim X(0)+X^3(T).
 \end{align*}
\end{Prop}
\begin{proof}
The estimate of $\mathcal{M}(t)$ is divided in two parts.\\\\
\textbf{i) Control of $\mathcal{M}_s(t)$, with $s \ge 3+\tau$.}
We consider the equations for the unknowns $(b, z)$ in system \eqref{eq:system-z2}, %which we recall below 
%\begin{equation}
%\left\{
%\begin{aligned}
    %&\d_t b - \varepsilon\cR_1^2 b  = \cR_1 z + (\cR_2 \Omega, - \cR_1 \Omega) \cdot \nabla b, \\
    %&\d_t z + \frac{z}{\varepsilon}=-\varepsilon \cR_1^2 \Omega - \varepsilon \cR_1 ((\cR_2 \Omega, -\cR_1\Omega) \cdot \nabla b) + \Lambda^{-1} ((\cR_2 \Omega, -\cR_1\Omega) \cdot (\nabla \Lambda \Omega)). \label{eq:system-z2}
%\end{aligned}
%\right.
%\end{equation}
applying the operator $\La^s=(-\Delta)^\frac{s}{2}$ to each term with the notation $(\dot b, \dot z, \dot \Omega)=(\La^s b, \La^s z, \La^s \Omega)$.

% Notice that the first nonlinear term in the equation for $z$ in \eqref{eq:system-z2} can be written as follows 
% \begin{align*}
%   \La^{s} \cR_1 ((\cR_2\Omega, -\cR_1 \Omega)\cdot \nabla b)&=\La^{s-1}\d_x ((\cR_2\Omega, -\cR_1 \Omega)\cdot \nabla b)\\
%   &= (\cR_2\Omega, -\cR_1 \Omega)\cdot \nabla \Lambda^{-1} \d_x \dot b + [\La^{s-1}\d_x, \cR_2\Omega] \cdot \d_x b-[\La^{s-1}\d_x, \cR_1\Omega] \cdot \d_y b\\
%   &= (\cR_2\Omega, -\cR_1 \Omega)\cdot \nabla \cR_1 \dot b + [\La^{s-1}\d_x, \cR_2\Omega] \cdot \d_x b-[\La^{s-1}\d_x, \cR_1\Omega] \cdot \d_y b.
% \end{align*}

Noticing that $\Lambda^s\cR_1=\Lambda^{s-1}\d_x$, the quasi-linearized system reads
\begin{equation}\label{eq:quasilin}
    \left\{
    \begin{aligned}
        \d_t \dot b - \varepsilon\cR_1^2 \dot b & = \cR_1 \dot z + (\cR_2 \Omega, - \cR_1\Omega) \cdot \nabla \dot b + [\La^s, \cR_2 \Omega] \d_x b-[\La^s, \cR_1 \Omega] \d_y b \\
    &:=\cR_1 \dot z+\cI_1+\mathcal{C}_{1,1}-\mathcal{C}_{1,2}, \\
    \d_t \dot z + \frac{\dot z}{\varepsilon}&=-\varepsilon \cR_1^2 \dot \Omega - \varepsilon  ((\cR_2 \Omega, - \cR_1\Omega) \cdot \nabla \cR_1 \dot b) +  (\cR_2 \Omega, - \cR_1\Omega) \cdot  \nabla  \dot \Omega\\
    &\quad + [\La^{s-1}, \cR_2 \Omega] \d_x \Lambda \Omega -[\La^{s-1}, \cR_1 \Omega] \d_y \Lambda \Omega\\
    & \quad  - \varepsilon [\La^{s-1}\d_x, \cR_2 \Omega] \d_x b+\varepsilon  [\La^{s-1}\d_x, \cR_1 \Omega] \d_y b\\
    &=:-\varepsilon \cR_1^2 \dot \Omega- \cI_2+\cI_3 + \mathcal{C}_{2,1}-\mathcal{C}_{2,2}-\mathcal{C}_{3,1}+\mathcal{C}_{3,2}.  
    \end{aligned}
    \right.
\end{equation}
Similarly, the quasi-linearized equation for $\Omega$ (from \eqref{eq:system-newvariable}) reads
\begin{align}\label{eq:omega-dot}
    \d_t \dot \Omega + \frac{\dot \Omega}{\varepsilon}=\cR_1 \dot b + ((\cR_2\Omega, -\cR_1\Omega)\cdot (\nabla \dot \Omega))+ [\La^{s-1}, \cR_2 \Omega] \d_x \Lambda \Omega -[\La^{s-1}, \cR_1 \Omega] \d_y \Lambda \Omega.
\end{align}

\noindent Let us provide the desired estimate.%, by multiplying the equation for $\dot b$ in \eqref{eq:quasilin} by $\dot b$, the equation for $\dot z$ in \eqref{eq:quasilin} by $\dot z$ and the equation for $\dot \Omega$ against $\dot \Omega$.
\subsubsection*{The linear evolution}
\noindent We first look at the linear terms, neglecting the nonlinearity. Using the skew-symmetry of the Riesz transform $(\cR_1 \dot z, b)_{L^2}=-(\dot z, \cR_1 \dot b)_{L^2}$ and $(\cR_1^2 \dot \Omega, \dot z)_{L^2}=-(\cR_1 \dot \Omega, \cR_1 \dot z)_{L^2}$, 
\begin{align*}
    \frac 12 \frac{d}{dt} (\|b\|^2_{\dot H^s}+\|\Omega\|^2_{\dot H^s}+\|z\|^2_{\dot H^s}) + \varepsilon \|\cR_1 b\|^2_{\dot H^s}&+ \frac{1}{\varepsilon}\|z\|^2_{\dot H^s}+\frac{1}{\varepsilon}\|\Omega\|^2_{\dot H^s}=-(\cR_1 \dot b, \dot z)_{L^2}+\varepsilon (\cR_1 \dot \Omega, \cR_1 \dot z)_{L^2}.
\end{align*}
Using the Cauchy-Schwarz inequality
\begin{align*}
    |(\cR_1 \dot b, \dot z)_{L^2}| &\le  \frac{\varepsilon}{2}\|\cR_1 b\|_{\dot H^s}^2 + \frac{1}{2\varepsilon}\|z\|_{\dot H^s}^2,\\
    \varepsilon |(\cR_1 \dot \Omega, \cR_1 \dot z)_{L^2}| & \le \frac{\varepsilon^2}{2}\|\cR_1  b\|_{\dot H^s}^2 + \frac{1}{2} \| \cR_1  z\|_{\dot H^s}^2,
\end{align*}
and the continuity of the Riesz transform $\|\cR_j a\|_{L^2} \le \|a\|_{L^2}$ for any $a \in L^2$,
the last two terms are absorbed by the linear dissipation yielding the inequality
\begin{align*}
    \frac 12 \frac{d}{dt} (\|b\|^2_{\dot H^s}+\|\Omega\|^2_{\dot H^s}+\|z\|^2_{\dot H^s}) + \frac{\varepsilon}{4} \|\cR_1 b\|^2_{\dot H^s}&+ \frac{1}{4\varepsilon}\|z\|^2_{\dot H^s}+\frac{1}{\varepsilon}\|\Omega\|^2_{\dot H^s} \le 0.
\end{align*}

Now we deal with the nonlinearity of system \eqref{eq:quasilin}.

\subsubsection*{Estimates of the nonlinear term}
Consider the quasi-linearized terms $\mathcal{I}_j$ for $j\in \{1,2,3\}$.
The divergence-free condition and integration by parts yield
\begin{align*}
    (\cI_1, \dot b)_{L^2} = ((\cR_2 \Omega, - \cR_1\Omega) \cdot \nabla \dot b, \dot b)_{L^2}=0. 
\end{align*}
Furthermore, it is easy to see that
\begin{align*}
    -\mathcal{I}_2+\mathcal{I}_3=(\cR_2 \Omega, - \cR_1\Omega) \cdot  \nabla \dot z,
\end{align*}
readily implying, using again the divergence free condition, that
    $(-\mathcal{I}_2+\mathcal{I}_3, \dot z)=0.$
Next, let us focus on the commutators.
    
\subsubsection*{Commutator estimates}
We begin with the terms of type $(\mathcal{C}_{2, j}, \dot z)_{L^2}, (\mathcal{C}_{3, j}, \dot z)_{L^2}$ for $j \in \{1,2\}$, which are easier being quadratic in (some norm of) the dissipative variable $z$. Notice indeed that as the products $(\mathcal{C}_{1, j}, \dot b)_{L^2}$ are quadratic in $b$, they require a more careful treatment since the $L^2_T$ control of the (spatial norm of the) variable $b$ is not provided by the energy functional (one only controls $\cR_1 b$, see the definition of $\mathcal{M}_\tau (T)$ in \eqref{def:Mmu}).

\noindent Let us first look at $ \mathcal{C}_{2,1}=[\Lambda^{s-1}, \mathcal{R}_2 \Omega] \partial_x \Lambda \Omega$. 
By applying the commutator estimates in Lemma \ref{lem:comm} (in the Appendix) with $s\ge 3+\tau>\frac{d}{2}+1$, one has 
\begin{align*}
\| \mathcal{C}_{2,1}\|_{L^2}\ & \lesssim \| \nabla \cR_2 \Omega\|_{L^\infty}\|\cR_1\Omega\|_{{\dot H}^s}+\|\d_x \Lambda \Omega\|_{L^\infty} \|\cR_2 \Omega\|_{{\dot H}^{s-1}}.
\end{align*} 
Now, applying first Lemma \ref{AnisoEmbed} and after Lemma \ref{BesovSobolev} with $1-\tau=\tau_1 < 2 < s=\tau_2$, we have 
\begin{align*}
 \| \nabla \cR_2 \Omega\|_{L^\infty} \lesssim \|\cR_2 \Omega\|_{B^{\frac 32, \frac 12}} \lesssim  \|\cR_2 \Omega\|_{\dot H^{1-\tau} \cap \dot H^{s}}.
\end{align*}
\footnote{Note that the controls $\|\Omega\|_{\dot H^{2-\tau}}, \|b\|_{\dot H^{2-\tau}}$ would be enough and there is no need of $\|\Omega\|_{\dot H^{1-\tau}}, \|b\|_{\dot H^{1-\tau}}$ at this stage. }
Similarly, using again Lemma \ref{AnisoEmbed} and Lemma \ref{BesovSobolev} 
\begin{align*}
\| \partial_x \Lambda \Omega\|_{L^\infty} \lesssim \|\La^2 \cR_1 \Omega\|_{L^\infty} \lesssim \| \Lambda \Omega\|_{B^{\frac 12, \frac 32}} \lesssim \|\Lambda \Omega\|_{\dot H^{2-\tau} \cap \dot H^{2+\tau}} \lesssim \| \Omega\|_{\dot H^{3-\tau} \cap \dot H^{s}}, \quad (s\ge 3+\tau),
\end{align*}
where now, using the interpolation Lemma \ref{lem:int}
\begin{align*}
\|\Omega\|_{\dot H^{3-\tau}} \lesssim \|\Omega\|_{\dot H^{1-\tau}}^\theta \|\Omega\|_{\dot H^s}^{1-\theta}, \quad \text{with} \quad \theta=\frac{s+\tau-3}{s+\tau-1}. 
\end{align*}
Finally, the Young inequality $ab \le \frac{a^p}{p}+\frac{b^q}{q}$ with $p=\frac 1\theta, q=\frac{1}{1-\theta}$ gives
\begin{align*}
\|\Omega\|_{\dot H^{3-\tau}} \lesssim \  \|\Omega\|_{\dot H^{1-\tau}} + \|\Omega\|_{\dot H^{s}}.
\end{align*}
In the same way, using that $\|\cR_1 a\|_{\dot H^s} \le \|a\|_{\dot H^s}$ for any $s\ge 0$ and $a \in \dot H^s$, by interpolation 
\begin{align*}
\|\cR_2 \Omega\|_{\dot H^{s-1}} \lesssim \|\Omega\|_{\dot H^{s-1}} \lesssim \|\Omega\|_{\dot H^{1-\tau}}^{\tilde \theta} \|\Omega\|_{\dot H^{s}}^{\tilde \theta} \lesssim   \|\Omega\|_{\dot H^{1-\tau}} + \|\Omega\|_{\dot H^{s}}, \quad \text{with} \quad \tilde \theta=\frac{1}{s+\tau-1}. 
\end{align*}
Altogether, it yields 
\begin{align*}
    \int_0^T |(\mathcal{C}_{2,1}, \dot z)_{L^2}|\, dt & \le \int_0^T \| \mathcal{C}_{2,1}\|_{L^2}\|z\|_{{\dot H}^s} \, dt \lesssim \int_0^T \|\Omega\|_{\dot H^{1-\tau} \cap \dot H^s} \|\Omega\|_{\dot H^{1-\tau} \cap \dot H^s} \|z\|_{\dot H^s} \, dt\\
&\lesssim (\|\Omega\|_{L^\infty_T({\dot H}^{1-\tau})}+\|\Omega\|_{L^\infty_T(\dot H^s)})(\|\Omega\|_{L^2_T({\dot H}^{1-\tau})}+\|\Omega\|_{L^2_T(\dot H^s)}) \|z\|_{L^2_T({\dot H}^s)} \\
&\lesssim  \mathcal{M}^3(T). 
\end{align*}

The commutator $\mathcal{C}_{2,2}$ is completely analogous, we omit it. Let us consider $\mathcal{C}_{3,1}=\varepsilon [\Lambda^{s-1}\d_x, \cR_2 \Omega] \d_x b$, which gives
\begin{align*}
\|\mathcal{C}_{3,1}\|_{L^2}& \lesssim \varepsilon (\| \nabla \cR_2 \Omega\|_{L^\infty}\|\cR_1 b\|_{{\dot H}^s}+\|\partial_x b\|_{L^\infty} \|\cR_1 \cR_2 \Omega\|_{{\dot H}^s}).
\end{align*}
Now note that, applying the same reasoning as before (Lemma \ref{AnisoEmbed}, Lemma \ref{BesovSobolev}),
\begin{align*}
\|\partial_x b\|_{L^\infty} \lesssim \|\Lambda \cR_1 b\|_{B^{\frac 12, \frac 12}} \lesssim \|\Lambda \cR_1 b\|_{\dot H^{1-\tau} \cap \dot H^{s-1}} \lesssim \|\cR_1 b\|_{\dot H^{2-\tau} \cap \dot H^{s}}, 
\end{align*}
where, again,
\begin{align*}
\|\cR_1 b\|_{\dot H^{2-\tau}} \lesssim \|\cR_1 b\|_{\dot H^{1-\tau}}^{\theta}  \|\cR_1 b\|_{\dot H^{s}}^{1-\theta} \lesssim \|\cR_1 b\|_{\dot H^{1-\tau}} + \|\cR_1 b\|_{\dot H^{s}}, \quad \theta=\frac{s+\tau-2}{s+\tau-1}.
\end{align*}

This way
\begin{align*}
    \int_0^T |(\mathcal{C}_{3,1}, \dot \Omega)_{L^2}| \, dt & \lesssim \varepsilon \int_0^T (\| \nabla \cR_2 \Omega\|_{L^\infty}\|\cR_1 b\|_{{\dot H}^s}+\|\d_x b\|_{L^\infty} \|\cR_2 \Omega\|_{{\dot H}^s})\|\Omega\|_{\dot H^s} \, dt \\
& \lesssim \varepsilon \int_0^T \|\Omega\|_{\dot H^{1-\tau} \cap \dot H^s}^2 \|b\|_{\dot H^s} + \|\cR_1 b\|_{\dot H^{1-\tau} \cap \dot H^s} \|\Omega\|_{\dot H^s}^2 \, dt \\
& \lesssim \varepsilon^2 \|b\|_{L^\infty_T (\dot H^{1-\tau} \cap \dot H^s)}\times \frac{\|\Omega\|^2_{L^2_T (\dot H^{1-\tau} \cap \dot H^s)}}{\varepsilon} \\
&\lesssim \mathcal{M}^3(T).
\end{align*}
The commutator $\mathcal{C}_{3,2}$ is bounded similarly. 

 Now we deal with the more delicate estimates involving the commutators $\mathcal{C}_{1,j}$ for $j \in \{1,2\}$ in the equation of $\dot b$. Let us start with the term with $\mathcal{C}_{1,1}=[\Lambda^s, \cR_2\Omega] \partial_x b$, where it is easy to reconstruct the term $\cR_1 \dot b$, which is controlled in $L^2_T$ by the energy functional $\mathcal{M}(T)$. We have, applying the above reasoning, that
\begin{align*}
\|\mathcal{C}_{1,1}\|_{L^2} &\lesssim \|\nabla \cR_2 \Omega\|_{L^\infty} \|\partial_x b\|_{\dot H^{s-1}} + \|\cR_2 \Omega\|_{\dot H^s} \|\partial_x b\|_{L^\infty} \\
&\lesssim \|\Omega\|_{\dot H^{1-\tau} \cap \dot H^s} \|\cR_1 b\|_{\dot H^s} + \|\cR_1 b\|_{\dot H^{1-\tau} \cap \dot H^s} \|\Omega\|_{\dot H^s},
\end{align*}
so that
\begin{align*}
    \int_0^T |(\mathcal{C}_{1,1}, \dot b)_{L^2}| \, dt & \le \int_0^T \|\mathcal{C}_{1,1}\|_{L^2} \|b\|_{{\dot H}^s} \, d t  \lesssim \int_0^T \frac{1}{\sqrt \varepsilon} \|\Omega\|_{\dot H^{1-\tau} \cap \dot H^s} \|\cR_1 b\|_{\dot H^{1-\tau} \cap \dot H^s} \sqrt \varepsilon\|b\|_{{\dot H}^s}\, dt  \\
    & \lesssim \frac{1}{\sqrt \varepsilon}\|\Omega\|_{L_T^2 ({\dot H}^{1-\tau} \cap \dot H^s)}\|b\|_{L^\infty_T( {\dot H}^s)}  \sqrt \varepsilon \| \cR_1 b\|_{L^2_T ({\dot H}^{1-\tau}\cap \dot H^s)} \lesssim \mathcal{M}^3(T). 
\end{align*}
The next commutator $\mathcal{C}_{1,2}=[\Lambda^s, \cR_1 \Omega] \partial_y b$ requires a more careful treatment.
We rely on the fractional Leibniz rule, which is an extension, due to Li \cite{Li} (see also D'Ancona \cite{Dancona}) of the Kenig-Ponce-Vega inequality to the case $s \ge 1$.\\
 We introduce the notation $\alpha=(\alpha_1, \alpha_2) \in \mathbb{N}^2 \, (\beta=(\beta_1, \beta_2) \in \mathbb{N}^2)$ and $\nabla^\alpha=(\partial_x^{\alpha_1}, \partial_y^{\alpha_2})$, while the operator $\Lambda^{s, \alpha}$ is defined via Fourier transform as
\begin{align}
\widehat{\Lambda^{s, \alpha} f}(\xi) = \widehat{\Lambda^{s, \alpha}}(\xi) \widehat{f}(\xi), \qquad
\widehat{\Lambda^{s, \alpha}}(\xi)=i^{-|\alpha|}\partial_\xi^\alpha (|\xi|^s),
\end{align}
where $|\alpha|=\alpha_1+\alpha_2$ (resp. $|\beta|=\beta_1+\beta_2$).
Notice that $\Lambda^{s,\alpha}$ is a pseudo-differential operator of order $s-\alpha.$
Now, we apply Lemma \ref{lem:Lin}, with $s_1=1, s_2=s-1$, which gives
\begin{align}\label{est:KPV}
\left\|\mathcal{C}_{1,2} -\sum_{|\alpha|=1} \frac{1}{\alpha!} \nabla^\alpha \cR_1 \Omega \Lambda^{s, \alpha} \partial_y b - \sum_{|\beta| \le s-2}\frac{1}{\beta!}  \nabla^\beta \partial_y b \Lambda^{s, \beta} \cR_1 \Omega \right\|_{L^2} \lesssim \|\Lambda \cR_1 \Omega\|_{\text{BMO}} \|\partial_y b\|_{\dot H^{s-1}}.
\end{align}

Then we write the scalar product adding and subtracting the above right-hand side as follows
\begin{align*}
(\mathcal{C}_{1,2}, \dot b)_{L^2} &= \left(\mathcal{C}_{1,2} -\sum_{|\alpha|=1} \frac{1}{\alpha!} \nabla^\alpha \cR_1 \Omega \Lambda^{s, \alpha} \partial_y b - \sum_{|\beta| \le s-2}\frac{1}{\beta !}  \nabla^\beta \partial_y b \Lambda^{s, \beta} \cR_1 \Omega, \dot b\right)_{L^2} \\
&\quad +\sum_{|\alpha|=1} \frac{1}{\alpha!} (\nabla^\alpha \cR_1 \Omega \Lambda^{s, \alpha} \partial_y b, \dot b)_{L^2} + \sum_{|\beta| \le s-2}\frac{1}{\beta !}  (\nabla^\beta \partial_y b \Lambda^{s, \beta} \cR_1 \Omega, \dot b)_{L^2},
\end{align*}
yielding, thanks to \eqref{est:KPV},
%%%%%
\begin{align}\label{est:C21}
|(\mathcal{C}_{1,2}, \dot b)_{L^2}|& \lesssim  \|\Lambda \cR_1 \Omega\|_{\text{BMO}} \|\partial_y b\|_{\dot H^{s-1}}\|b\|_{\dot H^s}+|\mathcal{J}_1|+|\mathcal{J}_2|,
\end{align}
where we denote
\begin{align*}
\mathcal{J}_1:&= \sum_{|\alpha|=1} \frac{1}{\alpha !}(\nabla^\alpha \cR_1 \Omega \Lambda^{s, \alpha} \partial_y b, \dot b)_{L^2} \andf
 \mathcal{J}_2:= \sum_{|\beta| \le s-2} \frac{1}{\beta !} (\nabla^\beta \partial_y b \Lambda^{s, \beta} \cR_1 \Omega, \dot b)_{L^2}.
\end{align*}

Now note that $\mathcal{J}_1$ rewrites as 
\begin{align*}
\sum_{|\alpha|=1} \nabla^\alpha \cR_1 \Omega \Lambda^{s, \alpha} \partial_y b &=- s(\nabla  \cR_1 \Omega)( \Lambda^{s-2} \nabla \partial_y b) = - s(\nabla  \cR_1 \Omega)( \Lambda^{s-1} \nabla \mathcal{R}_2 b),
\end{align*}

so that using the continuity of $\cR_2$ in $L^2$
\begin{align}\label{est:J1}
|\mathcal{J}_1|& \lesssim \|\nabla \cR_1 \Omega\|_{L^\infty} \|\Lambda^{s} \cR_2 b\|_{L^2} \|b\|_{\dot H^s} \lesssim \|\nabla \cR_1 \Omega\|_{L^\infty}\|b\|_{\dot H^s}^2,
\end{align} 
and integrating in time, recalling that $u_2=-\cR_1 \Omega$,
\begin{align*}
\int_0^T |\mathcal{J}_1| \, dt \lesssim  \|\nabla \cR_1 \Omega\|_{L^1_T (L^\infty)}\|b\|_{L^\infty_T( \dot H^s)}^2 \lesssim \|\nabla u_2 \|_{L^1_T( L^\infty)}\mathcal{M}_s(T) \lesssim X(T)^3. 
\end{align*}

Now let us deal with $\mathcal{J}_2$.
As $\cR_1= \Lambda^{-1}{\partial_x}$, we integrate by parts in the horizontal direction
\begin{align}
\mathcal{J}_2=\sum_{|\beta|\le s-2} ((\nabla^\beta \partial_y b) \Lambda^{s-1,\beta} \partial_x\Omega , \dot b)_{L^2}&=-\sum_{|\beta|\le s-2} ((\nabla^\beta \partial^2_{xy} b) \Lambda^{s-1,\beta} \Omega , \dot b)_{L^2}\notag\\
&\quad -\sum_{|\beta|\le s-2} ((\nabla^\beta \partial_y b) \Lambda^{s-1,\beta}\Omega , \partial_x \dot b)_{L^2}. \label{eq:J2}
\end{align}

Now, we deal with the last term in \eqref{eq:J2}. Using that $\partial_x \dot b=\Lambda (\partial_x \Lambda^{s-1} b)$ and the symmetry of $\Lambda$, it can be written as 
\begin{align*}
\sum_{|\beta| \le s-2} ((\nabla^\beta \partial_y b) \Lambda^{s-1,\beta}\Omega , \partial_x \dot b)_{L^2}&= \sum_{|\beta| \le s-2} (\Lambda (\nabla^\beta \partial_y b \Lambda^{s-1, \beta}  \Omega), \partial_x \Lambda^{s-1}  b)_{L^2}\\
&=\sum_{|\beta| \le s-2}(\Lambda (\nabla^\beta \partial_y b \Lambda^{s-1, \beta}  \Omega), \cR_1 \dot b)_{L^2}.
\end{align*}
We use the following decomposition
\begin{align*}
\sum_{|\beta| \le s-2}(\Lambda (\nabla^\beta \partial_y b \Lambda^{s-1, \beta}  \Omega), \cR_1 \dot b)_{L^2}&=(\Lambda ( \partial_y b \Lambda^{s-1}  \Omega), \cR_1 \dot b)_{L^2}\\
&\quad +\sum_{|\beta| =1 \, \text{and} \, \beta=(1,1)}(\Lambda (\nabla^\beta \partial_y b \Lambda^{s-1, \beta}  \Omega), \cR_1 \dot b)_{L^2} \\
&\quad  + \sum_{2 \le |\beta|\le s-2 \, \text{and} \, \beta\neq(1,1)}(\Lambda (\nabla^\beta \partial_y b \Lambda^{s-1, \beta}  \Omega), \cR_1 \dot b)_{L^2}\\
&=\mathcal{J}_{2}^a + \mathcal{J}_{2}^b+\mathcal{J}_{2}^c.
\end{align*}

By virtue of the product Lemma \ref{lem:product}, one obtains 
\begin{align*}
\mathcal{J}_{2}^a & \lesssim \|\cR_1 b\|_{\dot H^s} (\|\Lambda \partial_y b\|_{L^\infty} \|\Lambda^{s-1}\Omega\|_{L^2} + \|\partial_y b\|_{L^\infty} \|\Lambda^{s}\Omega\|_{L^2} ) \qquad \text{($\beta=(0,0)$)}\\
\mathcal{J}_2^b & \lesssim \|\cR_1 b\|_{\dot H^s} (\|\Lambda \nabla \partial_y b\|_{L^2} \|\Lambda^{s-3}\nabla\Omega\|_{L^\infty} + \|\nabla \partial_y b\|_{L^\infty} \|\Lambda^{s-2} \nabla \Omega\|_{L^2}) \qquad \text{($|\beta|=1$ and $\beta=(1,1)$)}\\
\mathcal{J}_2^c & \lesssim \|\cR_1 b\|_{\dot H^s} \sum_{2 \le |\beta| \le s-2} (\|\Lambda \nabla^\beta \partial_y b\|_{L^2} \|\Lambda^{s-1, \beta} \Omega\|_{L^\infty} + \|\nabla^\beta \partial_y b\|_{L^2} \|\Lambda^{s, \beta} \Omega\|_{L^\infty}) \qquad \text{($|\beta|\ge 2$)}.
\end{align*}
\medbreak
Concerning the first term of $\mathcal{J}_2$ in \eqref{eq:J2}, one has
\begin{align*}
    \sum_{|\beta|\le s-2} |((\nabla^\beta \partial^2_{xy} b) \Lambda^{s-1,\beta} \Omega , \dot b)_{L^2}|& \lesssim \sum_{|\beta|\le s-2}\|(\nabla^\beta \partial^2_{xy} b) \Lambda^{s-1,\beta} \Omega\|_{L^2} \|b\|_{\dot H^s}.
\end{align*}
For $\beta=(0,0)$, we have
\begin{align*}
    \|(\partial_{xy}^2 b) \La^{s-1} \Omega\|_{L^2} & \lesssim \|\partial_{xy} b\|_{L^\infty} \|\Omega\|_{\dot H^{s-1}}\lesssim \|\cR_1 b\|_{\dot H^{3-\tau} \cap \dot H^{3+\tau}} \|\Omega\|_{\dot H^{s-1}}.
    \end{align*}
The remaining terms yield
    \begin{align*}
    \sum_{|\beta|=1 \, \text{and} \, \beta=(1,1)} |((\nabla^\beta \partial^2_{xy} b) \Lambda^{s-1,\beta} \Omega , \dot b)_{L^2}|& \lesssim \sum_{|\beta|=1 \, \text{and} \, \beta=(1,1)}\|\nabla^\beta \partial_{xy}^2 b\|_{L^2} \|\La^{s-1, \beta} \Omega\|_{L^\infty}\|b\|_{\dot H^s}\\& \lesssim \|b\|_{\dot H^3} \|\Omega\|_{\dot H^{3-\tau} \cap \dot H^{3+\tau}}\|b\|_{\dot H^s}
    \end{align*}
and
\begin{align*}
     \sum_{2\le |\beta|\le s-2} |((\nabla^\beta \partial^2_{xy} b) \Lambda^{s-1,\beta} \Omega , \dot b)_{L^2}|& \lesssim \sum_{2\le |\beta|\le s-2} \|\nabla^\beta \partial^2_{xy} b\|_{L^2} \|\La^{s-1, \beta} \Omega\|_{L^\infty}\|b\|_{\dot H^s}.
\end{align*}

Altogether, appealing to the embeddings of Lemma \ref{AnisoEmbed} and Lemma \ref{BesovSobolev} ($s \ge 3+\tau$), one obtains
\begin{align*}
\mathcal{J}_{2} & \lesssim (\|\cR_1 b\|_{\dot H^{1-\tau}}+\|\cR_1 b\|_{\dot H^s}) (  \|b\|_{\dot H^{1-\tau}}+ \|b\|_{\dot H^{s}})  (  \|\Omega\|_{\dot H^{1-\tau}}+ \|\Omega\|_{\dot H^{s}}),
\end{align*}
so that
\begin{align*}
\int_0^T |\mathcal{J}_2| \, dt & \lesssim (\|\cR_1 b\|_{L^2_T (\dot H^{1-\tau})}+\|\cR_1 b\|_{L^2_T (\dot H^s)})  (  \|b\|_{L^\infty_T (\dot H^{1-\tau})}+ \|b\|_{L^\infty_T (\dot H^{s})})\\
&\quad \times (  \|\Omega\|_{L^2_T (\dot H^{1-\tau})}+ \|\Omega\|_{L^2_T (\dot H^{s})}) \\&\lesssim \mathcal{M}(T)^3.
\end{align*}

Inserting the latter in \eqref{est:C21} together with \eqref{est:J1} and using the embedding $L^\infty \hookrightarrow \text{BMO}$ yields  
\begin{align*}
\int_0^T |(\mathcal{C}_{1,2}, \dot b)_{L^2}|\, dt & \lesssim (\|\Lambda \cR_1 \Omega\|_{L^1_T (L^\infty)} + \|\nabla \cR_1 \Omega\|_{L^1_T( L^\infty)})\mathcal{M}^2(T)+\mathcal{M}^3(T).
\end{align*}

To control the terms $\|\nabla \cR_1\Omega\|_{L^\infty}, \|\Lambda \cR_1 \Omega\|_{L^1_T L^\infty} $, we shall rely on anisotropic Besov spaces in Step II (Section \ref{sec:Linftycontrol}).
\\\\
%%%%%
\textbf{ii) Control of $\mathcal{M}_{1-\tau}(t)$,  $0<\tau < 1$.}
We apply $\La^{1-\tau}$ to system \eqref{eq:system-z2}, yielding
\begin{equation}\label{eq:system-tau}
\left\{
\begin{aligned}
      \d_t \dot b -\varepsilon \cR_1^2 \dot b &= \cR_1 \dot z+ \La^{1-\tau}((\cR_2\Omega, -\cR_1\Omega)\cdot \nabla b), \\\d_t \dot z + \frac{\dot z}{\varepsilon}&=-\varepsilon \cR_1^2 \dot \Omega - \varepsilon \La^{1-\tau}\cR_1((\cR_2\Omega, -\cR_1\Omega)\cdot \nabla b)+\La^{-\tau} ((\cR_2\Omega, -\cR_1\Omega)\cdot (\nabla \La \Omega)).
\end{aligned}
\right.
\end{equation}
Moreover, the equation for $\dot \Omega$ reads
\begin{align*}
    \d_t \dot \Omega + \frac{\dot \Omega}{\varepsilon}&=\La^{-\tau} ((\cR_2\Omega, -\cR_1\Omega)\cdot (\nabla \La \Omega)).
\end{align*}
The linear terms work exactly the same way as before, thus we focus on the nonlinearities. We begin with
\begin{align*}
(\La^{1-\tau}((\cR_2\Omega, -\cR_1\Omega)\cdot \nabla b), \dot b)_{L^2}&=: \mathcal{I}_1+\mathcal{I}_2.
\end{align*}
Let us start with $\mathcal{I}_1$. Appealing to the product estimate (Lemma \ref{lem:product}) with $p_1=r_1=\infty$ yields 
\begin{align*}
     \mathcal{I}_1&=(\La^{1-\tau}(\cR_2\Omega \d_x b), \dot b)_{L^2}  \lesssim ({\|\cR_2 \Omega\|_{L^\infty} \|\d_x b\|_{\dot H^{1-\tau}}+\| \Lambda^{1-\tau}\cR_2 \Omega\|_{L^2} \|\d_x b\|_{L^\infty})\|b\|_{\dot H^{1-\tau}}},
\end{align*}
%\footnote{Here we use $\|\cR_2\Omega\|_{L^\infty} \lesssim \|\cR_2 \Omega\|_{B^\frac 12, \frac 12} \lesssim \|\cR_2 \Omega\|_{\dot H^{1-\tau} \cap \dot H^{s}}$. This is a passage where }
where interpolation (Lemma \ref{lem:int}) and Young inequality give
\begin{align*}
\|\d_x b\|_{\dot H^{1-\tau}} \le \|\cR_1 b\|_{\dot H^{2-\tau}} \lesssim \|\cR_1 b\|_{\dot H^{1-\tau}}^\theta \|\cR_1 b\|_{\dot H^{s}}^{1-\theta}, \quad \theta=\frac{s+\tau-2}{s+\tau-1}.
\end{align*}
This yields 
\begin{align*}
    \int_0^T   \mathcal{I}_1 \, dt & \lesssim \sqrt \varepsilon(\|\cR_1 b\|_{L^2_T( \dot H^{1-\tau})} +\|\cR_1 b\|_{L^2_T( \dot H^{s})})   \|b\|_{L^\infty_T (\dot H^{1-\tau})} \frac{1}{\sqrt \varepsilon} (\|\Omega\|_{L^2_T( \dot H^{1-\tau})}+\|\Omega\|_{L^2_T (\dot H^s)}) \\&\lesssim \mathcal{M}^3(T).
\end{align*}
We cannot use the same trick for the next term, which gives 
\begin{align*}
     \int_0^T   \mathcal{I}_2 \, dt & \lesssim  \| b\|_{L^\infty_T (\dot H^{1-\tau})}  (\| b\|_{L^\infty_T (\dot H^s)}+ \|b\|_{L^\infty_T (\dot H^{1-\tau})})  \| \cR_1 \Omega\|_{L^1_T( W^{1,\infty})}\\&\lesssim \| u_2 \|_{L^1_T( W^{1,\infty})}\mathcal{M}^2(T) \lesssim X^3(T).
\end{align*}
Next, it is easy to see that $\varepsilon \cR_1 (\cR_2\Omega, -\cR_1 \Omega) \cdot \nabla b$ is very similar to the terms treated before, then it is omitted.

It remains to deal with the last term in the equation of $\dot z$ (which, in the energy estimate, is multiplied both by $\dot z$ and by $\dot \Omega$, but the computations are identical, so that we only detail one case). Using the symmetry of the multiplier $\La^{-\tau}$, one obtains  
\begin{align*}
    \mathcal{I}_3:&=(\La^{-\tau} ((\cR_2\Omega, -\cR_1\Omega)\cdot (\nabla \La \Omega)), \dot z)_{L^2}  = 
    ((\cR_2\Omega, -\cR_1\Omega)\cdot (\nabla \La \Omega), \La^{1-2\tau} z)_{L^2}.
    \end{align*}
Now, we integrate by parts in $x$ the first addend and in $y$ the second one. As the term $\cR_2\partial_x \Omega - \cR_1 \partial_y \Omega=0$ (from the divergence-free condition), it remains 
\begin{align*}
    \mathcal{I}_3
    &=- (\cR_2 \Omega \, \La \Omega , \La^{1-2\tau} \partial_x z)_{L^2} +(\cR_1 \Omega \, \La \Omega, \La^{1-2\tau} \partial_y z)_{L^2}\\
    &= - (\cR_2 \Omega \, \La \Omega, \La^{2(1-\tau)} \cR_1 z)_{L^2} +(\cR_1 \Omega \, \La \Omega, \La^{2(1-\tau)} \cR_2 z)_{L^2}\\
    &= - (\La^{1-\tau}(\cR_2 \Omega \, \La \Omega) , \La^{1-\tau} \cR_1 z)_{L^2} +(\La^{1-\tau}(\cR_1 \Omega \, \La \Omega), \La^{1-\tau} \cR_2 z)_{L^2}\\
    &=:\mathcal{I}_3^a+\mathcal{I}_3^b.
    \end{align*}

This way
\begin{align*}
|\mathcal{I}_3^a|& \le \|\La^{1-\tau}(\cR_2 \Omega \, \La \Omega) \|_{L^2} \|z\|_{\dot H^{1-\tau}},
\end{align*}
and using the product Lemma \ref{lem:product}, Lemma \ref{AnisoEmbed} and Lemma \ref{BesovSobolev},
\begin{align*}
 \|\La^{1-\tau}(\cR_2 \Omega \, \La \Omega) \|_{L^2} &\lesssim \|\cR_2 \Omega\|_{\dot H^{1-\tau}} \| \La \Omega\|_{L^\infty}+ \|\cR_2 \Omega\|_{L^\infty}\| \Lambda\Omega\|_{\dot H^{1-\tau}}\\
  %& \lesssim \|\cR_2 \Omega\|_{\dot H^{1-\tau}} \|\La \Omega\|_{\dot H^{1-\tau} \cap \dot H^{2+\tau}}+ \|\cR_2 \Omega\|_{\dot H^{1}}\|\Omega\|_{\dot H^{2-\tau}}
  %\\
%  & \lesssim \|\cR_2 \Omega\|_{\dot H^{1-\tau}} \|\La \Omega\|_{\dot H^{1-\tau} \cap \dot H^{2+\tau}}+ \|\cR_2 \Omega\|_{\dot H^{1-\tau}\cap \dot H^{1+\tau}}\sum_{j \in \{1,2\}}\|\cR_j \cR_2 \Omega\|_{\dot H^{2-\tau}}\\
 & \lesssim (\|\Omega\|_{\dot H^{1-\tau}}+\|\Omega\|_{\dot H^{s}})^2.
\end{align*}
The estimate of $\mathcal{I}_3^b$ is identical.
Finally,
\begin{align*}
    \int_0^T \mathcal{I}_3 \, dt &\lesssim   \varepsilon (\|\Omega\|_{L^\infty_T (\dot H^{1-\tau})}+ \|\Omega\|_{L^\infty_T( \dot H^{s})}) \frac{1}{\sqrt \varepsilon} (\|\Omega\|_{L^2_T (\dot H^{1-\tau})}+\|\Omega\|_{L^2_T (\dot H^{s})}) \frac{1}{\sqrt \varepsilon} \|z\|_{L^2_T (\dot H^{1-\tau})} \\&\lesssim \varepsilon \mathcal{M}^3(T). 
\end{align*}
The proof of Proposition \ref{Prop:apriori-Sobolev} is concluded.

\end{proof}

 \subsection{II. $L^\infty$ and Lipschitz bounds for $u_2=\cR_1 \Omega$}\label{sec:Linftycontrol}
 The purpose of this section is to prove the following proposition.
\begin{Prop} \label{Propu2}
For $\varepsilon>0$, let $(b,z)$ be a smooth solution of \eqref{eq:system-newvariable}. One has
\begin{align*}
    \|(\nabla \cR_1\Omega,\Lambda \cR_1\Omega)\|_{L^1_T(L^\infty)}+\|\cR_1\Omega\|_{L^1_T(L^\infty)}\lesssim X(0)+ X(t)^2.
\end{align*}
\end{Prop}
%%%%%

 \subsubsection{Linear a priori estimates in anisotropic spaces} \label{sec:BesovPart}
 
 Applying $\ddj\ddq^h$ to the linear part of \eqref{eq:system-z2}, we obtain
\begin{equation}
\left\{
\begin{aligned}
    &\d_t b_{j,q} - \varepsilon\cR_1^2 b_{j,q}  = \cR_1z_{j,q} +h_{j,q}, \\
    &\d_t z_{j,q} + \frac{z_{j,q}}{\varepsilon}=-\varepsilon\cR_1^2\Omega_{j,q} +g_{j,q}, \label{LinearBesov1}
\end{aligned}
\right.
\end{equation}
where $h=(\cR_2 \Omega, - \cR_1 \Omega) \cdot \nabla b$ and $g=-\varepsilon \cR_1 (\cR_2 \Omega, -\cR_1\Omega) \cdot \nabla b + \Lambda^{-1} ((\cR_2 \Omega, -\cR_2\Omega) \cdot (\nabla \Lambda \Omega))$.
Performing standard energy estimates, one obtains
\begin{align}\label{BeforeBern}
   \frac{1}{2} \dfrac{d}{dt}\|b_{j,q}\|_{L^2}^2+\varepsilon\|\cR_1 b_{j,q}\|_{L^2}^2\leq \|\cR_1 z_{j,q}\|_{L^2}\|b_{j,q}\|_{L^2}+ \|h_{j,q}\|_{L^2}\|b_{j,q}\|_{L^2}.
\end{align}
Using Fourier-Plancherel theorem and the anisotropic Bernstein inequality in Lemma \ref{AnisoBernstein} yields
\begin{align}\label{Bern}
   \frac{1}{2} \dfrac{d}{dt}\|b_{j,q}\|_{L^2}^2+2^{-2j}2^{2q}\varepsilon\| b_{j,q}\|_{L^2}^2\leq \|\cR_1z_{j,q}\|_{L^2}\|b_{j,q}\|_{L^2}+\|h_{j,q}\|_{L^2}\|b_{j,q}\|_{L^2}.
\end{align}
Now, applying Lemma \ref{SimpliCarre}, multiplying by $2^{js_1}2^{qs_2}$ and summing on $j,k\in \Z$ give
\begin{align}\label{eq:bBesov01}
    \|b\|_{L^\infty_T(B^{s_1,s_2})}+ \varepsilon\|b\|_{L^1_T(B^{s_1-2,s_2+2})} \lesssim \|b_0\|_{B^{s_1,s_2}}+\|z\|_{B^{s_1-1,s_2+1}} +\|h\|_{L^1_T(B^{s_1,s_2})}
\end{align} Following a similar procedure for $z$, one infers
\begin{align}\label{eq:zBesov02}
    \|z\|_{L^\infty_T(B^{s_1',s_2'})}+ \frac{1}{\varepsilon}\|z\|_{L^1_T(B^{s_1',s_2'})} \lesssim \|z_0\|_{B^{s_1',s_2'}}+\varepsilon\|\Omega\|_{B^{s_1-2,s_2+2}} +\|g\|_{L^1_T(B^{s_1',s_2'})}.
\end{align}

Notice that the linear term $\|z\|_{B^{s_1-1,s_2+1}}$ in \eqref{eq:bBesov01} can be absorbed by the left-hand side of \eqref{eq:zBesov02} if $s_1'\geq s_1$ and $s_2'\leq s_2$ via Lemma \ref{InclusBesov} with $s=1$.
The linear term $\|\Omega\|_{B^{s_1-2,s_2+2}}$ will be absorbed in a similar fashion once the estimates for $\Omega$ are obtained.

Since $\Omega=z-\varepsilon\cR_1b$, this linear analysis together with  Lemma \ref{AnisoEmbed} suggests us to choose
\begin{align*}
    &s_1=s_1'=\frac32\andf s_2=s_2'=\frac12, \text{ to ensure that one controls}  \|\nabla \cR_1\Omega\|_{L^1_T(L^\infty)}, 
    \\& s_1=s_1'=\frac12 \andf \:s_2=s_2'=\frac12,\text{ to ensure that one controls }  \| \cR_1\Omega\|_{L^1_T(L^\infty)}.
\end{align*}
\begin{Rmq}
The term $\| \cR_1\Omega\|_{L^1_T(L^\infty)}$ is not needed to close the a priori estimates in Sobolev spaces (proof of Proposition \ref{Prop:apriori-Sobolev}), but it is actually crucial to deal with some nonlinear terms appearing in this anisotropic analysis. The control of $\| \cR_1\Omega\|_{L^1_T(W^{1,\infty})}$ requires a priori estimates in two different regularity settings (as just remarked above).
\end{Rmq}
% \begin{align*}
%     &s_1=s_1'=\frac32,\andf s_2=s_2'=\frac12 \andf s_1'=\frac12, \: s_2'=\frac32 \implies \text{we control }  \|\nabla \cR_1\Omega\|_{L^1_T(L^\infty)} 
%     \\&  s_1=\frac12,\:s_2=\frac12 \andf s_1'=-\frac12, \: s_2'=\frac32 \implies \text{we control } \| \cR_1\Omega\|_{L^1_T(L^\infty)}.
% \end{align*}
With these regularity indexes, one obtains
\begin{align*}
    &\|b\|_{L^\infty_T(B^{\frac32,\frac12})}+ \varepsilon\|b\|_{L^1_T(B^{-\frac12,\frac52})}+\sqrt{\varepsilon}\|b\|_{L^2_T(B^{\frac12,\frac32})}  \lesssim \|b_0\|_{B^{\frac32,\frac12}} +\|h\|_{L^1_T(B^{\frac32,\frac12})},
\\
   & \|b\|_{L^\infty_T(B^{\frac12,\frac12})}+ \varepsilon\|b\|_{L^1_T(B^{-\frac32,\frac52})}+\sqrt{\varepsilon}\|b\|_{L^2_T(B^{-\frac12,\frac32})} \lesssim \|b_0\|_{B^{\frac12,\frac12}} +\|h\|_{L^1_T(B^{\frac12,\frac12})},
\end{align*}
and
\begin{align*}
   & \|z\|_{L^\infty_T(B^{\frac32,\frac12})}+ \frac1\varepsilon\|z\|_{L^1_T(B^{\frac32,\frac12})} +\frac{1}{\sqrt{\varepsilon}}\|z\|_{L^2_T(B^{\frac32,\frac12})}  \lesssim\|z_0\|_{B^{\frac32,\frac12}} +\|g\|_{L^1_T(B^{\frac32,\frac12})},
   \\&
    \|z\|_{L^\infty_T(B^{\frac12,\frac12})}+ \frac1\varepsilon\|z\|_{L^1_T(B^{\frac12,\frac12})}+\frac{1}{\sqrt{\varepsilon}}\|z\|_{L^2_T(B^{\frac12,\frac12})} \lesssim \|z_0\|_{B^{\frac12,\frac12}} +\|g\|_{L^1_T(B^{\frac12,\frac12})},
\end{align*}
where we used interpolation inequalities to recover the $L^2_T$ terms. Since $\Omega=z-\varepsilon\cR_1b$, one derives the following bounds
\begin{align*}
    &\|\Omega\|_{L^\infty_T(B^{\frac32,\frac12})}+\|\Omega\|_{L^1_T(B^{\frac12,\frac32})} +\|\Omega\|_{L^2_T(B^{\frac32,\frac12})}  \lesssim \|(\Omega_0,b_0)\|_{B^{\frac32,\frac12}} +\|(g,h)\|_{L^1_T(B^{\frac32,\frac12})},
\\&
    \|\Omega\|_{L^\infty_T(B^{\frac12,\frac12})}+\|\Omega\|_{L^1_T(B^{-\frac12,\frac32})} +\|\Omega\|_{L^2_T(B^{\frac12,\frac12})}  \lesssim\|(\Omega_0,b_0)\|_{B^{\frac12,\frac12}} +\|(g,h)\|_{L^1_T(B^{\frac12,\frac12})}.
\end{align*}
As expected, via Lemma \ref{AnisoEmbed} one has
\begin{align}
  \|\nabla \cR_1\Omega\|_{L^1_T(L^\infty)} \lesssim  \|\Omega\|_{L^1_T(B^{\frac12,\frac32})}\andf
    \|\cR_1 \Omega\|_{L^1_T(L^\infty)}\lesssim \|\Omega\|_{L^1_T(B^{-\frac12,\frac32})}.
\end{align}
Accordingly, we define the functional
\begin{align}\label{YFunc}
    Y(t)&= \|(b,z,\Omega)\|_{L^\infty_T(B^{\frac12,\frac12}\cap B^{\frac32,\frac12})}+\varepsilon \|b\|_{L^1_T(B^{-\frac32,\frac52}\cap B^{-\frac12,\frac52})}+\sqrt{\varepsilon}\|b\|_{L^2_T(B^{-\frac12,\frac32}\cap B^{\frac12,\frac32})}
    \notag\\ &\quad +\frac1\varepsilon \|z\|_{L^1_T(B^{\frac32,\frac12}\cap B^{\frac12,\frac12})}+\frac{1}{\sqrt{\varepsilon}}\|z\|_{L^2_T(B^{\frac32,\frac12}\cap B^{\frac12,\frac12})}\notag\\
    &\quad 
   + \|\Omega\|_{L^1_T(B^{\frac12,\frac32}\cap B^{-\frac12,\frac32})}+\|\Omega\|_{L^2_T(B^{\frac32,\frac12}\cap B^{\frac12,\frac12})},
\end{align} 
which will be useful to close the estimates.

 \subsubsection{Estimates of the nonlinearities}
 
The term
 $\|(h,g)\|_{L^1_T(B^{\frac32,\frac12}\cap B^{\frac12,\frac12})}$ is bounded by the following lemma.
 \begin{Lemme} \label{NLControl}
Let $(b,\Omega)$ be a smooth solution of \eqref{eq:system-newvariable}, then one has the following estimate:
\begin{align*}
   \|(h,g)\|_{L^1_T(B^{\frac32,\frac12}\cap B^{\frac12,\frac12})} \lesssim Y(t)^2+Y(t)X(t)+X(t)^2
\end{align*}
\end{Lemme}
The proof of Lemma \ref{NLControl} is postponed to Section \ref{sec:NL-PL} and it is based on product laws in anisotropic Besov spaces (again, postponed to  Section \ref{sec:NL-PL}).

\begin{proof}[Proof of Proposition \ref{Propu2}]
It is a direct consequence of Lemma \ref{NLControl}.
\end{proof}

%%%%%

 \subsection{Conclusion of the proof of Theorem \ref{Thm:Exist}}
 
 Gathering the estimates from Proposition \ref{Prop:apriori-Sobolev} and \ref{Propu2}, we obtain
 \begin{align}
     X(t)\lesssim X(0)+X(t)^3+X(t)^2+Y(t)X(t)+Y(t)^2.
 \end{align}
 Then, by Lemma \ref{BesovSobolev}, one has $$Y(t)\lesssim X(t),$$ which yields
  \begin{align}
     X(t)\lesssim X(0)+X(t)^2+X(t)^3.
 \end{align}
 From there, a standard bootstrap argument leads to the existence of global-in-time solutions of \eqref{eq:system-newvariable}. Then the uniqueness follows from stability estimate below and Theorem \ref{Thm:Exist} is proven.
 
 \subsection{Uniqueness: Stability estimate}
Let $(\rho_1,\rho_2)$ be two solutions of $\eqref{eq:IPM}$ associated to the same initial data. We define $w=\rho_1-\rho_2$, it satisfies
\begin{equation}\label{eq:w}
    \d_tw-\cR^2_1w=(\cR_2 \cR_1 \rho_1 )\pa_x \rho_1 - (\cR_1^2 \rho_1) \pa_y \rho_1-(\cR_2 \cR_1 \rho_2) \pa_x \rho_2 + (\cR_1^2 \rho_2) \pa_y \rho_2.
\end{equation}
The right-hand side terms may be rewritten as follows:
% \begin{equation}\label{eq:w2}
%     \d_tw-\cR^2_1w=(\cR_2 \cR_1 w) \pa_x \rho_1+(\cR_2 \cR_1 \rho_2) \pa_x w - (\cR_1^2w)\d_y\rho_1-(\cR_1^2 \rho_2) \d_yw,
% \end{equation}
\begin{equation}\label{eq:w3}
    \d_tw-\cR^2_1w=(\cR_2 \cR_1 w,-\cR_1^2w)\cdot \nabla \rho_1+(\cR_2 \cR_1 \rho_2,-\cR_1^2\rho_2)\cdot \nabla w.
\end{equation}
Applying $\La^s=(-\Delta)^\frac{s}{2}$ to \eqref{eq:w3} and using the notation $\dot w=\La^s w$, one infers that
\begin{equation}\label{eq:w4}
    \d_t\dot{w}-\cR^2_1\dot{w}=\Lambda^{s}((\cR_2 \cR_1 w,-\cR_1^2w)\cdot \nabla \rho_1)+\Lambda^s((\cR_2 \cR_1 \rho_2,-\cR_1^2\rho_2)\cdot \nabla w).
\end{equation}
Combining the use of commutator estimates for the nonlinear terms (as in the previous section) with the bounds on $(\rho_1,\rho_2)$ from Theorem \ref{Thm:Exist} and the Gronwall inequality easily gives $w=0$, from which the uniqueness of smooth solutions for $\eqref{eq:IPM}$ follows. Analogous arguments lead to the uniqueness of smooth solutions to system \eqref{eq:system-newvariable}.

  \section{Proof of the relaxation limit Theorem \ref{Thm:Relax}}
Recall the scaled variables below: $$(\widetilde{b}^\varepsilon,\widetilde{\Omega}^\varepsilon)(\tau,x)
\triangleq(b,\frac{\Omega}{\varepsilon})(t,x)
\with \tau=\varepsilon t.$$
The system reads:
\begin{equation}\label{eq:boussEps}
\left\{
\begin{aligned}
    \pa_t \widetilde{b}^\varepsilon - \cR_1 \widetilde{\Omega}^\varepsilon & = (\cR_2 \widetilde{\Omega}^\varepsilon) \pa_x \widetilde{b}^\varepsilon - (\cR_1 \widetilde{\Omega}^\varepsilon) \pa_y \widetilde{b}^\varepsilon, \\
    \varepsilon^2\pa_t \widetilde{\Omega}^\varepsilon - \cR_1 \widetilde{b}^\varepsilon + \widetilde{\Omega}^\varepsilon&=\varepsilon^2\Lambda^{-1} [(\cR_2 \wt \Omega^\varepsilon, - \cR_1 \wt\Omega^\varepsilon) \cdot (\nabla \Lambda\wt\Omega^\varepsilon)].
\end{aligned}
\right.
\end{equation}

Let $(\wt b^\varepsilon,\wt \Omega^\varepsilon)$ be the unique solution of \eqref{eq:boussEps} from Theorem \ref{Thm:Exist}.
Under such rescaling it satisfies the following estimate:
 \begin{align} \label{eq:scalelimit}
&\|\wt b^\varepsilon\|_{L^\infty_T({\dot H^{1-\tau}\cap\dot H}^s)}+\varepsilon\|\wt \Omega^\varepsilon\|_{L^\infty_T(\dot H^{1-\tau}\cap{\dot H}^s)}+ \|\cR_1b\|_{L^2_T(\dot H^{1-\tau}\cap{\dot H}^s)}+\|\Omega\|_{L^2_T(\dot H^{1-\tau}\cap{\dot H}^s)}
\\&\nonumber\hspace{8cm}+ \|\wt\Omega^\varepsilon-\cR_1\wt b^\varepsilon\|_{L^2_T(\dot H^{1-\tau}\cap{\dot H}^s)}\leq \widetilde{\mathcal{M}}_0.
 \end{align}
Owing to \eqref{eq:scalelimit}, $\varepsilon \wt \Omega^\varepsilon$ and $\wt \Omega^\varepsilon$ (and therefore $\cR_1\Omega$ and $\cR_2 \Omega$) are uniformly bounded in the spaces
$L^\infty(\R^+;\dot H^{1-\tau}\cap\dot H^s)$ and $L^2(\R^+;\dot H^{1-\tau}\cap\dot H^s),$ respectively. This implies  that
$$\varepsilon^2\Lambda^{-1} [(\cR_2 \wt \Omega^\varepsilon, - \cR_1 \wt\Omega^\varepsilon) \cdot (\nabla \Lambda\wt\Omega^\varepsilon)]\quad\hbox{in }\ L^2(\R^+;\dot H^{s-1}).$$
Therefore $\varepsilon^2\d_t\wt \Omega^\varepsilon$ goes to $0$ in the sense of distributions. Putting this information into the second equation of \eqref{eq:boussEps}, one infers that
  \begin{equation}\label{eq:weakCvg}\wt \Omega^\varepsilon-\wt \cR_1b^\varepsilon \rightharpoonup 0 \ \text{  in  }\  \mathcal{D}'(\mathbb{R}^+\times\mathbb{R}^d).
  \end{equation}
Concerning the other unknown, 
$\widetilde{b}^\varepsilon$ is uniformly bounded in $L^\infty(\R^+;\dot H^{1-\tau}\cap \dot H^{s}).$ Therefore, there exists $\rho \in L^\infty(\R^+;\dot H^{1-\tau}\cap \dot H^{s})$ such that, up to subsequence, 
 \begin{equation}\label{eq:weakn}\wt{b}^\varepsilon\overset{\ast}{\rightharpoonup} \rho\ \text{ in  }\  L^\infty(\R^+;\dot H^{1-\tau}\cap \dot H^{s}).\end{equation}
Then, as the bounds from \eqref{eq:scalelimit} easily ensure bounds for the time-derivative of the solution $(\d_t \wt b^\varepsilon,\d_t \wt \Omega^\varepsilon)$, a standard procedure involving compactness argument and Aubin-Lions lemma leads to $\wt b^\varepsilon\to\rho$ strongly \footnote{At first, the convergence takes place only for a subsequence, then it is deduced for the whole sequence because the limit system has a unique solution so all the sequences will converge to the same limit.} in $C([0,T],{\dot H}^{1-\tau'}_{loc}\cap{\dot H}^{s-s'}_{loc})$ for $0<\tau<\tau'<1$ and $0<s'<s$.
For more details, we refer to Coulombel and Lin in \cite{CoulombelLin} or Xu and Wang in \cite{XuWang} for the relaxation limit of the compressible Euler system with damping in the inhomogeneous Sobolev and Besov settings respectively.
Now, defining
\begin{equation}\label{eq:Wepsilon}
\wt Z^\varepsilon:=\wt \Omega^\varepsilon-\wt \cR_1 b^\varepsilon
\end{equation} the first equation of \eqref{eq:boussEps} may be rewritten as
\begin{equation}\label{eq:safd}
\partial_t\wt b^\varepsilon-\cR_1^2\wt b^\varepsilon=
\wt S^\varepsilon\with \wt S^\varepsilon= \cR_1 \wt z^\varepsilon + (\cR_2 \wt \Omega^\varepsilon, - \cR_1 \wt \Omega^\varepsilon) \cdot \nabla \wt b^\varepsilon.
\end{equation}
Hence, combining \eqref{eq:scalelimit}, \eqref{eq:weakCvg}, \eqref{eq:weakn} and \eqref{eq:Wepsilon}, 
one can deduce that
\begin{equation}\label{CauchyPbM}\partial_t\rho-\cR_1^2\rho=-(\cR_2 \cR_1 \rho, - \cR_1^2 \rho) \cdot \nabla \rho\end{equation}
which concludes the proof of Theorem \ref{Thm:Relax}.

   \appendix
     \bigbreak

  \section{} \label{sec:NL-PL}
  \begin{Lemme}[Anisotropic product laws] \label{AnisoPL}
Let $-\frac12\leq s_1 \leq \frac52$ and $-\frac12 \leq s_2 \leq \frac12$. Let $\delta_1,\delta_2,\delta_3,\delta_4\geq 0$ and $f,g$ be two smooth functions, one has
\begin{align*}
   \|\ddj\ddq^h (fg)\|_{B^{s_1,s_2}}&\lesssim \|f\|_{L^\infty}\|g\|_{B^{s_1,s_2}}+\|f\|_{B^{s_1,s_2}}\|g\|_{L^\infty} \\&\quad +\|f\|_{B^{s_1+\frac12+\delta_1,-\delta_2}}\|g\|_{B^{-\delta_1,s_2+\frac12+\delta_2}}+  \|f\|_{B^{-\delta_3,s_2+\frac12+\delta_4}}  \|g\|_{B^{s_1+\frac12+\delta_3,-\delta_4}}.
\end{align*}
Moreover, 
\begin{align*}
   \|\ddj\ddq^h (fg)\|_{B^{-\frac12,\frac12}}&\lesssim \|f\|_{L^\infty}\|g\|_{B^{-\frac12,\frac12}}+\|f\|_{B^{\frac12,\frac12}}\|\Lambda^{-1} g\|_{L^\infty} \\&\quad +\|f\|_{B^{\delta_1,-\delta_2}}\|g\|_{B^{-\delta_1,1+\delta_2}}+  \|f\|_{B^{-\delta_3,1+\delta_4}}  \|g\|_{B^{\delta_3,-\delta_4}}.
\end{align*}
\end{Lemme}
    \subsection{Proof of Lemma \ref{AnisoPL}}
    \begin{proof}[Proof of Lemma \ref{AnisoPL}]
    Recalling the definitions of the Littlewood-Paley blocks $\dot \Delta_j$ and $\dot\Delta_q^h$ in Section \ref{sec:mainr}, we introduce the isotropic and anisotropic paradifferential decomposition due to Bony in \cite{Bony}.
Let $f,g\in\cS'(\R^d)$, 
$$fg=T(f,g)+R(f,g) \andf fg=T^h(f,g)+R^h(f,g) $$
where 
$$T(f,g)= \sum_{j\in \mathbb{Z}}S_{j-1}f\ddj g \andf R(f,g)=\sum_{j\in \mathbb{Z}}\ddj f S_{j+2}g, $$
  and in the horizontal direction
  $$T^h(f,g)= \sum_{q\in \mathbb{Z}}S_{q-1}f\ddq^h g \andf R^h(f,g)=\sum_{q\in \mathbb{Z}}\ddq^h f S_{q+2}g. $$
Applying the double paraproduct decomposition to $fg$, one has
\begin{align}
    fg=TT^h(f,g)+TR^h(f,g)+RT^h(f,g)+RR^h(f,g).
\end{align}
Let us estimate each of these terms separately. For the first one, we have
\begin{align*}
   \|\ddj\ddq^h TT^1(f,g)\|_{L^2}&\lesssim \sum_{|j'-j|\leq 4, |q'-q|\leq 4} \|S_{j'-1}S_{q'-1}f\dot{\Delta}_{j'}\dot{\Delta}_{q'}\d_yg\|_{L^2}
   \\&\lesssim  \sum_{|j'-j|\leq 4, |q'-q|\leq 4} \|S_{j'-1}S_{q'-1}f\|_{L^\infty}\|\dot{\Delta}_{j'}\dot{\Delta}_{q'}\d_yg\|_{L^2}
   \\& \lesssim c_{j,q}2^{-js_1}2^{-qs_2}\|f\|_{L^\infty}\|g\|_{B^{s_1,s_2}}.
\end{align*}
Concerning the fourth term, one has
\begin{align*}
   \|\ddj\ddq^h RR^1(f,g)\|_{L^2}&\lesssim \sum_{|j'-j|\leq 4, |q'-q|\leq 4} \|\dot{\Delta}_{j'}\dot{\Delta}_{q'}fS_{j'+2}S_{q'+2}g\|_{L^2}
   \\&\lesssim \sum_{|j'-j|\leq 4, |q'-q|\leq 4} \|\dot{\Delta}_{j'}\dot{\Delta}_{q'}f\|_{L^2}\|S_{j'+2}S_{q'+2}g\|_{L^\infty}
   \\&\lesssim  c_{j,q}2^{-js_1}2^{-qs_2}\|f\|_{B^{s_1,s_2}}\|g\|_{L^\infty}.
\end{align*}

The following computations in the special case $s_1=-\frac12$ and $s_2=\frac12$ will be needed in the proof of Lemma \ref{AnisoPL}:
\begin{align*}
   \|\ddj\ddq^h RR^1(f,g)\|_{L^2}&\lesssim \sum_{|j'-j|\leq 4, |q'-q|\leq 4} \|\dot{\Delta}_{j'}\dot{\Delta}_{q'}fS_{j'+2}S_{q'+2}g\|_{L^2}
   \\&\lesssim \sum_{|j'-j|\leq 4, |q'-q|\leq 4} \|\dot{\Delta}_{j'}\dot{\Delta}_{q'}f\|_{2}\|S_{j'+2}S_{q'+2}g\|_{L^\infty}
    \\&\lesssim \sum_{|j'-j|\leq 4, |q'-q|\leq 4}2^{j}\|\dot{\Delta}_{j'}\dot{\Delta}_{q'}f\|_{L^2}2^{-j}\|S_{j'+2}S_{q'+2}g\|_{L^2}
   \\&\lesssim  c_{j,q}2^{\frac j 2}2^{-\frac q 2}\|f\|_{B^{\frac12,\frac12}}\|\Lambda^{-1} g\|_{L^\infty}.
\end{align*}
For the third term, using the anisotropic Bernstein inequality, one has
\begin{align*}
   \|\ddj\ddq^h& TR^1(f,g)\|_{L^2}\\&\lesssim \sum_{|j'-j|\leq 4, |q'-q|\leq 4} \|\dot{\Delta}_{j'}S_{q'+2}fS_{j'+2}\dot{\Delta}_{q'}g\|_{L^2}
    \\&\lesssim \sum_{|j'-j|\leq 4, |q'-q|\leq 4} \|\dot{\Delta}_{j'}S_{q'+2}f\|_{L^2_x L^\infty_y}\|S_{j'+2}\dot{\Delta}_{q'}g\|_{L^\infty_x L^2_y }
    \\&\lesssim \sum_{|j'-j|\leq 4, |q'-q|\leq 4} 2^{\frac{j'}{2}} \|\dot{\Delta}_{j'}S_{q'+2}f\|_{L^2}2^{\frac{q'}{2}}\|S_{j'+2}\dot{\Delta}_{q'}g\|_{L^2}
    \\&\lesssim 2^{-js_1}2^{-qs_2}\sum_{|j'-j|\leq 4, |q'-q|\leq 4}  2^{j'(s_1+\frac12)} \|\dot{\Delta}_{j'}S_{q'+2}f\|_{L^2}2^{q'(s_2+\frac12)}\|S_{j'+2}\dot{\Delta}_{q'}g\|_{L^2}
    \\&\lesssim 2^{-js_1}2^{-qs_2}\sum_{|j'-j|\leq 4, |q'-q|\leq 4}  2^{j'(s_1+\frac12+\delta_1)}2^{-q\delta_2} \|\dot{\Delta}_{j'}S_{q'+2}f\|_{L^2}2^{q'(s_2+\frac12+\delta_2)}2^{-j'\delta_1}\|S_{j'+2}\dot{\Delta}_{q'}g\|_{L^2}
      \\&\lesssim2^{-js_1}2^{-qs_2} c_{j,q}
\left\{
\begin{aligned}
      &\|f\|_{B^{s_1+\frac12+\delta_1,-\delta_2}}\|g\|_{B^{-\delta_1,s_2+\frac12+\delta_2}} \quad &\text{if } \delta_1,\delta_2>0
     \\  &\|f\|_{B^{s_1+\frac12}}\|g\|_{B^{s_2+\frac12}_h} \quad &\text{if } \delta_1,\delta_2=0,
\end{aligned}
\right.
\end{align*}
where $B_h^s$ refers to the Besov spaces with horizontal localisation $\ddq$. Since 
\begin{align*}
    \|f\|_{B^{s_1+\frac12}}\|g\|_{B^{s_2+\frac12}_h}\lesssim \|f\|_{B^{s_1+\frac12,0}}\|g\|_{B^{0,s_2+\frac12}},
\end{align*}
we obtain the desired estimate. Similar computations lead to
\begin{align*}
   \|\ddj\ddq^h& RT^1(f,g)\|_{L^2}
 \lesssim 2^{-js_1}2^{-qs_2} c_{j,q} \|f\|_{B^{-\delta_3,s_2+\frac12+\delta_4}}  \|g\|_{B^{s_1+\frac12+\delta_3,-\delta_4}}.
\end{align*}
Multiplying the above estimates by $2^{js_1}2^{qs_2}$ and summing on $j,q\in\Z$, the desired result follows.
\end{proof}

 \subsection{Proof of Lemma \ref{NLControl}}
 \begin{proof}[Proof of Lemma \ref{NLControl}]
 
 \textbf{I) Estimates for $h=(\cR_2 \Omega, - \cR_1 \Omega) \cdot \nabla b$.}
 \medbreak
 
\noindent \textbf{i)} Applying Lemma \ref{AnisoPL} with $f=\cR_1 \Omega$, $g=\d_yb$, $s_1=\frac 32$, $s_2=\frac12$, $\delta_1=\delta_2=\frac12$ and $\delta_3=\delta_4=0$, one has
\begin{align*}
   \|\cR_1\Omega\d_yb\|_{L^1_T(B^{\frac32,\frac12})}&\lesssim \|\cR_1\Omega\|_{L^1_T(L^\infty)}\|\d_yb\|_{L^\infty_T(B^{\frac32,\frac12})}+\|\cR_1\Omega\|_{L^1_T(B^{\frac32,\frac12})}\|\d_yb\|_{L^\infty_T(L^\infty)}\\&\quad+   \|\cR_1\Omega\|_{L^2_T(B^{\frac52,-\frac12})}\|\d_yb\|_{L^2_T(B^{-\frac12,\frac32})} +\|\cR_1\Omega\|_{L^1_T(B^{0,1})}\|\d_yb\|_{L^\infty_T(B^{2})}.
\end{align*}

Now, using Lemma \ref{BesovSobolev} and \ref{InclusBesov} we have
\begin{itemize}
    \item $\|\cR_1\Omega\|_{L^1_T(L^\infty)}\|\d_yb\|_{L^\infty_T(B^{\frac32,\frac12})} \lesssim X(t)\|b\|_{L^\infty({\dot H}^s\cap H^{3-\varepsilon})} \lesssim X(t)^2$,
    \bigbreak
    \item $\|\cR_1\Omega\|_{L^1_T(B^{\frac32,\frac12})}\|\d_yb\|_{L^\infty_T(L^\infty)} \lesssim Y(t)\|b\|_{L^\infty(H^{2+\varepsilon})} \lesssim Y(t)X(t)$,
    \bigbreak
    \item $\|\cR_1\Omega\|_{L^2_T(B^{\frac52,-\frac12})}\|\d_yb\|_{L^2_T(B^{-\frac12,\frac32})}\lesssim \|\Omega\|_{L^2_T(B^{\frac32,\frac12})}\|b\|_{L^2_T(B^{\frac12,\frac32})} \lesssim Y(t)^2$,
    \bigbreak
    \item $\|\cR_1\Omega\|_{L^1_T(B^{0,1})}\|\d_yb\|_{L^\infty_T(B^{2})}\lesssim \|\Omega\|_{L^1_T(B^{-1,2})}\|b\|_{L^\infty_T(B^{3})}\lesssim \|\Omega\|_{L^1_T(B^{-\frac12,\frac32})}\|b\|_{L^\infty_T(B^{3})} \lesssim X(t)^2$.
\end{itemize}
Gathering the above estimates, one obtains
\begin{align*}
   \|\ddj\ddq^h \cR_1\Omega\d_yb\|_{L^1_T(B^{\frac32,\frac12)}}\lesssim X(t)^2.
\end{align*}

\noindent \textbf{ii)} Applying Lemma \ref{AnisoPL} with $f=\cR_1 \Omega$, $g=\d_yb$, $s_1=\frac 12$, $s_2=\frac12$, $\delta_1=\delta_2=\frac12$ and $\delta_3=\delta_4=0$, we get
\begin{align*}
   \|\cR_1\Omega\d_yb\|_{L^1_T(B^{\frac12,\frac12)}}&\lesssim \|\cR_1\Omega\|_{L^1_T(L^\infty)}\|\d_y b\|_{L^\infty_T(B^{\frac12,\frac12})}+\|\cR_1\Omega\|_{L^1_T(B^{\frac12,\frac12})}\|\d_yb\|_{L^\infty_T(L^\infty)}\\&\quad +   \|\cR_1\Omega\|_{L^2_T(B^{\frac32,-\frac12})}\|\d_yb\|_{L^2_T(B^{-\frac12,\frac32})} +\|\cR_1\Omega\|_{L^1_T(B^{0,1})}\|\d_yb\|_{L^\infty_T(B^{1})}.
\end{align*}

Using Lemma \ref{BesovSobolev} and \ref{InclusBesov}, we have
\begin{itemize}
    \item $\|\cR_1\Omega\|_{L^1_T(L^\infty)}\|\d_yb\|_{L^\infty_T(B^{\frac12,\frac12})} \lesssim Y(t)\|b\|_{L^\infty_T({\dot H}^s\cap H^{2-\tau})} \lesssim X(t)^2$,
    \bigbreak
    \item $\|\cR_1\Omega\|_{L^1_T(B^{\frac12,\frac12})}\|\d_yb\|_{L^\infty_T(L^\infty)} \lesssim Y(t)\|b\|_{L^\infty(H^{2+\tau})} \lesssim Y(t)X(t)$,
    \bigbreak
    \item $\|\cR_1\Omega\|_{L^2_T(B^{\frac32,-\frac12})}\|\d_yb\|_{L^2_T(B^{-\frac12,\frac32})}\lesssim \|\Omega\|_{L^2_T(B^{\frac12,\frac12})}\|b\|_{L^2_T(B^{\frac12,\frac32})} \lesssim Y(t)^2$,
    \bigbreak
    \item $\|\cR_1\Omega\|_{L^1_T(B^{0,1})}\|\d_yb\|_{L^\infty_T(B^{1})}\lesssim \|\Omega\|_{L^1_T(B^{-1,2})}\|b\|_{L^\infty_T(B^{2})}\lesssim \|\Omega\|_{L^1_T(B^{-\frac12,\frac32})}\|b\|_{L^\infty_T(B^{2})} \lesssim X(t)^2$.
\end{itemize}

\noindent \textbf{iii)} For the second addend of $h$, applying Lemma \ref{AnisoPL} with $f=\cR_2 \Omega$, $g=\d_xb$, $s_1=\frac 32$, $s_2=\frac12$, $\delta_1=\delta_2=\frac12$ and $\delta_3=\delta_4=0$, one obtains
\begin{align*}
   \|\cR_2\Omega\d_yb\|_{L^1_T(B^{\frac32,\frac12)}}&\lesssim \|\cR_2\Omega\|_{L^2_T(L^\infty)}\|\d_xb\|_{L^2_T(B^{\frac32,\frac12})}+\|\cR_2\Omega\|_{L^2_T(B^{\frac32,\frac12})}\|\d_xb\|_{L^2_T(L^\infty)}\\&\quad +   \|\cR_2\Omega\|_{L^\infty_T(B^{\frac52,-\frac12})}\|\d_xb\|_{L^1_T(B^{-\frac12,\frac32})} +\|\cR_2\Omega\|_{L^2_T(B^{0,1})}\|\d_xb\|_{L^2_T(B^{2})}.
\end{align*}

Let us deal with each r.h.s. term:
\begin{itemize}
\item $\|\cR_2\Omega\|_{L^2_T(L^\infty)}\|\d_xb\|_{L^2_T(B^{\frac32,\frac12})}\lesssim \|\cR_2\Omega\|_{L^2_T(B^{\frac12,\frac12)}}\|\cR_1b\|_{L^2_T({\dot H}^s\cap H^{3-\tau)}} \lesssim  Y(t)X(t),$
\bigbreak
\item $\|\cR_2\Omega\|_{L^2_T(B^{\frac32,\frac12})}\|\d_xb\|_{L^2_T(L^\infty)}\lesssim Y(t)X(t),$
\bigbreak
\item $\|\cR_2\Omega\|_{L^\infty_T(B^{\frac52,-\frac12})}\|\d_xb\|_{L^1_T(B^{-\frac12,\frac32})}\lesssim \|\cR_2\Omega\|_{L^\infty_T(B^{\frac52,-\frac12})}\|b\|_{L^1_T(B^{-\frac12,\frac52})}\lesssim X(t)Y(t),$
\bigbreak
\item $\|\cR_2\Omega\|_{L^2_T(B^{0,1})}\|\d_xb\|_{L^2_T(B^{2})}\lesssim \|\cR_2\Omega\|_{L^2_T(B^{\frac12,\frac12})}\|\d_xb\|_{L^2_T(B^{2})}\lesssim \|\Omega\|_{L^2_T(B^{\frac12,\frac12})}\|\cR_1b\|_{L^2_T(B^{3})}\lesssim Y(t)X(t)$.
\end{itemize}

\noindent \textbf{iv)} Applying Lemma \ref{AnisoPL} with $f=\cR_2 \Omega$, $g=\d_xb$, $s_1=\frac 12$, $s_2=\frac12$, $\delta_1=\delta_2=0$ and $\delta_3=\delta_4=0$, 
\begin{align*}
   \|\cR_2\Omega\d_xb\|_{L^1_T(B^{\frac12,\frac12)}}&\lesssim \|\cR_2\Omega\|_{L^2_T(L^\infty)}\|\d_xb\|_{L^2_T(B^{\frac12,\frac12})}+\|\cR_2\Omega\|_{L^2_T(B^{\frac12,\frac12})}\|\d_xb\|_{L^2_T(L^\infty)}\\&\quad+   \|\cR_2\Omega\|_{L^2_T(B^{1,0})}\|\d_xb\|_{L^2_T(B^{0,1})} +\|\cR_2\Omega\|_{L^2_T(B^{0,1})}\|\d_xb\|_{L^2_T(B^{1})}.
\end{align*}
Again, we deal with each term separately:
\begin{itemize}
    \item $\|\cR_2\Omega\|_{L^2_T(L^\infty)}\|\d_xb\|_{L^2_T(B^{\frac12,\frac12})}\lesssim \|\Omega\|_{L^2_T(B^{\frac12,\frac12})}\|b\|_{L^2_T(B^{\frac12,\frac32})}\leq Y(t)^2$,
    \bigbreak
    \item $\|\cR_2\Omega\|_{L^2_T(B^{\frac12,\frac12})}\|\d_xb\|_{L^2_T(L^\infty)} \lesssim \|\Omega\|_{L^2_T(B^{\frac12,\frac12})}\|b\|_{L^2_T(B^{\frac12,\frac32})}\lesssim Y(t)^2$,
    \bigbreak
    \item $\|\cR_2\Omega\|_{L^2_T(B^{1,0})}\|\d_xb\|_{L^2_T(B^{0,1})}\lesssim \|\Omega\|_{L^2_T(B^{1})}\|\cR_1b\|_{L^2_T(B^{2})}\lesssim X(t)^2$,
    \bigbreak
    \item $\|\cR_2\Omega\|_{L^2_T(B^{0,1})}\|\d_xb\|_{L^2_T(B^{1})}\lesssim \|\Omega\|_{L^2_T(B^{1})}\|\cR_1b\|_{L^2_T(B^{2})}\lesssim X(t)^2$.
\end{itemize}
Gathering the above estimates yields
\begin{align*}
   \|\ddj\ddq^h \cR_1\Omega\d_yb\|_{L^1_T(B^{\frac12,\frac12})}&\lesssim X(t)^2.
\end{align*}

Adding the estimates from $\bold{i)-iv)}$ and using that $Y(t)\lesssim X(t)$ thanks to Lemma \ref{BesovSobolev}, one obtains
\begin{align}\label{est:h}
   \|h\|_{L^1_T(B^{\frac32,\frac12}\cap B^{\frac12,\frac12})}&\lesssim X(t)^2.
\end{align}

 \textbf{II) Estimates for $g=\varepsilon \cR_1 (\cR_2 \Omega, -\cR_1\Omega) \cdot \nabla b + \Lambda^{-1} ((\cR_2 \Omega, -\cR_1\Omega) \cdot (\nabla \Lambda \Omega)). $}
 
\medbreak
\noindent The estimates of the first term $\varepsilon \cR_1 (\cR_2 \Omega, -\cR_1\Omega) \cdot \nabla b$ can be obtained in the same way as the previous terms.
Indeed, notice that by Lemma \ref{InclusBesov} and thanks to the boundedness of the Riesz transform $\cR_1: B^{1,0}=B_{2,1}^1 \rightarrow B^{1,0}=B_{2,1}^1$, one has
\begin{align*}
    \|\cR_1 (\cR_2 \Omega \d_x b)\|_{L_T^1(B^{\frac 12, \frac 12})} &\lesssim \|\cR_1 (\cR_2 \Omega \d_x b)\|_{L_T^1(B^{1, 0})}\lesssim \|\cR_2 \Omega \d_x b\|_{L_T^1(B^{1, 0})},
\end{align*}
% Using  Lemma \ref{AnisoPL} with $\delta_1=\delta_2=\delta_3=\delta_4=0$ and  yield
% \begin{align*}
%     \|\cR_1 (\cR_2 \Omega \d_x b)\|_{L_T^1(B^{1, 0})} & \lesssim \|\cR_2 \Omega\|_{L^2_T( L^\infty)} \|\d_x b\|_{L^2_T (B^{1,0})}+\|\cR_2 \Omega\|_{L^2_T (B^{1,0})} \|\d_x b\|_{L^2_T( L^\infty)}\\
%     &\quad + \|\cR_2 \Omega\|_{L^2_T (B^{1,0})}\|\d_x b\|_{L^2_T( B^{0,1})}+\|\cR_2 \Omega\|_{L^2_T( B^{0,1})}\|\d_x b\|_{L^2_T( B^{1,0})},
% \end{align*}
which is exactly the same term as the one we treated in $\bold{I})$. A similar argument can be applied for the bound in $B^{\frac32,\frac12}$.
We then turn to the second addend of $g$.\\
\textbf{i)} First, observing that
\begin{align*}
    \|\Lambda^{-1} ((\cR_2 \Omega, -\cR_1\Omega) \cdot (\nabla \Lambda \Omega))\|_{B^{\frac32,\frac12}}\leq   \|(\cR_2 \Omega, -\cR_1\Omega) \cdot (\nabla \Lambda \Omega)\|_{B^{\frac12,\frac12}}
\end{align*}
and since $\Omega$ has better decay properties than $b$ and we control $\Omega$ in $L^2_T(H^s)$ for $s> 3$, we can directly deduce from our previous computations that 
\begin{align*}
    \|\Lambda^{-1} ((\cR_2 \Omega, -\cR_1\Omega) \cdot (\nabla \Lambda \Omega))\|_{B^{\frac32,\frac12}}\leq   X(t)^2.
\end{align*}

\textbf{ii)} For the second regularity setting it is a bit trickier as one has \begin{align*}
    \|\Lambda^{-1} ((\cR_2 \Omega, -\cR_1\Omega) \cdot (\nabla \Lambda \Omega))\|_{B^{\frac12,\frac12}}&\lesssim   \|(\cR_2 \Omega, -\cR_1\Omega) \cdot (\nabla \Lambda \Omega)\|_{B^{-\frac 12,\frac12}}.
    % \\&\lesssim  \|(\cR_2 \Omega, -\cR_11\Omega) \cdot (\nabla \Lambda \Omega)\|_{B^{0,0}}
\end{align*}
Applying the second inequality of Lemma \ref{AnisoPL} with $f=\cR_2 \Omega$, $g=\d_x\La \Omega$, $s_1=-\frac 12$, $s_2=\frac12$, $\delta_1=1$, $\delta_2=0$ and $\delta_3=\delta_4=0$, we obtain
\begin{align*}
   \|\cR_2\Omega\d_x\La \Omega\|_{L^1_T(B^{-\frac12,\frac12})}&\lesssim \|\cR_2\Omega\|_{L^\infty_T(L^\infty)}\|\d_x\La \Omega\|_{L^1_T(B^{-\frac12,\frac12})}+\|\cR_2\Omega\|_{L^\infty_T(B^{\frac12,\frac12})}\|\Lambda^{-1} \d_x\La \Omega\|_{L^1_T(L^\infty)}\\&\quad +   \|\cR_2\Omega\|_{L^2_T(B^{1})}\|\d_x\La \Omega\|_{L^2_T(B^{-1,1})} +\|\cR_2\Omega\|_{L^2_T(B^{0,1})}\|\d_x\La \Omega\|_{L^2_T(B^{0,0})}.
\end{align*}

\begin{Rmq}
Alternatively, one could think to exploit the first inequality of Lemma \ref{AnisoPL}, which would require a control of $\|\cR_2\Omega\|_{L^2_T(B^{-\frac12,\frac12})}\|\d_x\La \Omega\|_{L^2_T(L^\infty)}$. However, such bounds hold under additional low-regularity assumptions on the initial data, for instance $B^{0,0}$, that we want to avoid here.
\end{Rmq}
We estimate the above terms as follows.
\begin{itemize}
    \item $\|\cR_2\Omega\|_{L^\infty_T(L^\infty)}\|\d_x\La \Omega\|_{L^1_T(B^{-\frac12,\frac12})}\lesssim \|\Omega\|_{L^\infty_T(B^{\frac12,\frac12})}\|\Omega\|_{L^1_T(B^{\frac12,\frac32})}\lesssim Y(t)^2$,
    \bigbreak
    \item $\|\cR_2\Omega\|_{L^\infty_T(B^{\frac12,\frac12})}\|\Lambda^{-1} \d_x\La \Omega\|_{L^1_T(L^\infty)} \lesssim \|\Omega\|_{L^\infty_T(B^{\frac12,\frac12})}\|\Omega\|_{L^1_T(B^{\frac12,\frac32})}\lesssim X(t)Y(t)
    $,
    \bigbreak
    \item $\|\cR_2\Omega\|_{L^2_T(B^{1})}\|\d_x\La \Omega\|_{L^2_T(B^{-1,1})}\lesssim \|\Omega\|_{L^2_T(B^{1})}\|\Omega\|_{L^2_T(B^{0,2})}\lesssim \|\Omega\|_{L^2_T(B^{1})}\|\Omega\|_{L^2_T(B^{2,0})}\lesssim X(t)^2$,
    \bigbreak
    \item $\|\cR_2\Omega\|_{L^2_T(B^{0,1})}\|\d_x\La \Omega\|_{L^2_T(B^{0,0})}\lesssim \|\Omega\|_{L^2_T(B^{1})}\|\Omega\|_{L^2_T(B^{1,1})}\lesssim \|\Omega\|_{L^2_T(B^{1})}\|\Omega\|_{L^2_T(B^{\frac32,\frac12})}\lesssim X(t)Y(t)$.
\end{itemize}
The estimates for the last term $\cR_1\Omega\d_y\Lambda \Omega$ follow the exact same lines, then we omit them. %as $\cR_2$ and $\d_x$ do not play a special role in the above estimates as the unknown $b$ is not involved in this term.
We have therefore
\begin{align}\label{est:g}
   \|g\|_{L^1_T(B^{\frac32,\frac12}\cap B^{\frac12,\frac12})}&\lesssim X(t)^2.
\end{align}

Adding \eqref{est:h} and \eqref{est:g} concludes the proof of Lemma \ref{NLControl}.

 \end{proof}
    \section{Toolbox}
%  \subsection{Existence theorem for \eqref{eq:IPM}}
 
%  Here we justify briefly Theorem \ref{Prop:ExistIPM} as the proof follows exactly the same lines as the proof of Theorem \ref{Thm:Exist}.
 
  We collect some technical lemmas that are used in the course of the proof the results of this article. In some case, we provide (short) proofs, while in other cases, appealing to the existing literature to which we refer explicitly, we omit the proofs. When no explicit reference is provided being the results classical, the reader can look for instance at \cite{grafakos-book}. 
  
  \begin{Lemme}\label{SimpliCarre}
Let $X : [0,T]\to \mathbb{R}_+$ be a continuous function such that $X^2$ is differentiable. Assume that there exists 
 a constant $B\geq 0$ and  a measurable function $A : [0,T]\to \mathbb{R}_+$ 
such that 
 $$\frac{1}{2}\frac{d}{dt}X^2+BX^2\leq AX\quad\hbox{a.e.  on }\ [0,T].$$ 
 Then, for all $t\in[0,T],$ we have
$$X(t)+B\int_0^tX\leq X_0+\int_0^tA.$$
\end{Lemme}

\begin{Lemme}[Product estimates, \cite{WanBoussinesq}, Lemma 2.1]
\label{lem:product}
Let $s>0, 1 \le p,r \le \infty$, then
\begin{align*}
    \|\La^s(fg)\|_{L^p(\R^2)} \lesssim \| f\|_{L^{p_1}(\R^2)}\|\La^{s}g\|_{L^{p_2}(\R^2)}+\|g\|_{L^{r_1}(\R^2)}\|\La^s f\|_{L^{r_2}(\R^2)},
\end{align*}
where $1 \le p_1, r_1 \le \infty$ such that $\frac 1 p=\frac{1}{p_1}+\frac{1}{p_2}=\frac{1}{r_1}+\frac{1}{r_2}$.
\end{Lemme}
\begin{Lemme}[Commutator estimates, \cite{WanBoussinesq}, Lemma 2.1]\label{lem:comm}
Let $s>0$, $1 \le p_1, r_1 \le \infty$ and $1<p, p_1, r_1 <1$ such that $\frac 1p=\frac{1}{p_1} + \frac{1}{p_2} = \frac{1}{r_1} + \frac{1}{r_2}$. Then
\begin{align*}
\|[\La^s, f]g\|_{L^p} \lesssim (\|\nabla f\|_{L^{p_1}}\|\La^{s-1}g\|_{L^{p_2}}+\|\La^s f\|_{L^{r_1}}\|g\|_{L^{r_2}}).
\end{align*}
\end{Lemme}
%%%%%

\begin{Lemme}[Interpolation]
\label{lem:int}
Let $s_0 \le s \le s_1$. Then, for $\theta \in (0,1)$, such that $s=\theta s_0+(1-\theta)s_1$, it holds
\begin{align*}
    \|f\|_{\dot H^s} \lesssim \|f\|_{\dot H^{s_0}}^\theta  \|f\|_{\dot H^{s_1}}^{1-\theta}.
\end{align*}

\end{Lemme}

\begin{Lemme}[Embedding in Besov spaces] \label{InclusBesov}
For $s>0$, $s_1,s_2\in\R$, one has \begin{align}
B^{s_1+s,s_2-s} \subset B^{s_1,s_2}
\end{align}
\end{Lemme}
\begin{proof} Let $f\in B^{s_1+s,s_2-s}\cap B^{s_1,s_2}$. By definition of the localisation $\ddj$ and $\ddq^h$, when applying $\ddj\ddq^h$ to a function, one can use that there exists $N_0$ such that $j\geq q- N_0$.
This implies that
\begin{align*}
\sum_{j,q\in\Z} 2^{js_1}2^{qs_2}\|\ddj\ddq^hf\|_{L^2}&\lesssim  \sum_{j,q\in\Z} 2^{js_1}2^{js}2^{-js}2^{qs_2}\|\ddj\ddq^hf\|_{L^2}   
\\&\lesssim   \sum_{j,q\in\Z,j\geq q-N_0} 2^{js_1}2^{js}2^{-qs}2^{qs_2}\|\ddj\ddq^hf\|_{L^2}  
\\& \lesssim  \sum_{j,q\in\Z} 2^{j(s_1+s)}2^{q(s_2-s)}\|\ddj\ddq^hf\|_{L^2}.
\end{align*}
\end{proof}

%%%%%
The next lemma provides a generalized version of the Kenig-Ponce-Vega inequality (the fractional version of the Leibniz rule) for all $s>0$, see \cite{Li} and \cite{Dancona}. \\
Recall the notation $\alpha, \beta\in \mathbb{N}^2$ (multi-index) and $\nabla^\alpha=(\partial_x^{\alpha_1}, \partial_y^{\alpha_2})$, while the operator $\Lambda^{s, \alpha}$ is defined via Fourier transform as
\begin{align*}
\widehat{\Lambda^{s, \alpha} f}(\xi) = \widehat{\Lambda^{s, \alpha}}(\xi) \widehat{f}(\xi), \qquad
\widehat{\Lambda^{s, \alpha}}(\xi)=i^{-|\alpha|}\partial_\xi^\alpha (|\xi|^s).
\end{align*}

\begin{Lemme}[Generalized Kenig-Ponce-Vega inequality \cite{Li}, Theorem 5.1]\label{lem:Lin}
Let $s>0$ and $1<p<\infty$. Then, for any $s_1, s_2 \ge 0$ such that $s_1+s_2=s$, and any $f, g \in \mathcal{S}(\R^d)$, 
\begin{align*}
\left\|\La^s (fg)- \sum_{|\alpha| \le s_1} \frac{1}{\alpha !} (\nabla^\alpha f) (\La^{s,\alpha}g) - \sum_{|\beta| \le s_2-1} \frac{1}{\beta !} (\nabla^\beta g) (\La^{s,\beta}f) \right\|_{L^p} \lesssim \|\La^{s_1} f\|_{\text{BMO}} \|\La^{s_2} g\|_{L^p}. 
\end{align*}
\end{Lemme}

  \bigbreak
  \textbf{Conflicts of interest.}
 On behalf of all authors, the corresponding author states that there is no conflict of interest. 
 
\bigbreak
\textbf{Data availability statement.}
 Data sharing is not applicable to this article as no data sets were generated or analyzed during the current study.

\bibliographystyle{plain}

%\nocite{*}

\bibliography{Biblio}
\end{document}